\documentclass[english]{amsart}
\usepackage[latin1]{inputenc}
\usepackage{babel}

\usepackage{amsthm,amsmath,amssymb}
\usepackage{latexsym}
\usepackage{amscd}
\usepackage{mathrsfs}
\usepackage[all]{xy}
\usepackage{hyperref} 
\usepackage[usenames,dvipsnames]{color}
\usepackage{graphicx,eepic}
\usepackage{float}

\usepackage{color}

\theoremstyle{definition}
\newtheorem{thm}{Theorem}[section]
\newtheorem{prop}[thm]{Proposition}
\newtheorem{lem}[thm]{Lemma}
\newtheorem{conj}[thm]{Conjecture}
\newtheorem{cor}[thm]{Corollary}
\theoremstyle{definition}
\newtheorem{defn}[thm]{Definition}

\newtheorem{exa}[thm]{Example}
\theoremstyle{remark}
\newtheorem*{remark}{Remark}

\theoremstyle{definition}

\theoremstyle{definition}

\newcommand{\Fer}{\mathrm{Fer}}

\newcommand{\eqdef}{\stackrel{\rm def}{=}}

\newcommand{\Aut}{\mathrm{Aut}}

\newcommand{\qc}{G(r,p,q,n)}

\newcommand{\Irr}{\mathrm{Irr}}
\newcommand{\GCD}{\mathrm{GCD}}
\newcommand{\ZZ}{\mathbb{Z}}
\newcommand{\CC}{\mathbb{C}}
\newcommand{\cV}{\mathcal{V}}
\newcommand{\cR}{\mathcal{R}}
\def\End{\mathrm{End}}
\def\({\left(}
\def\){\right)}
\def\modu{\ (\mathrm{mod}\ }

\newcommand\spanning{\textnormal{-span}}

\def\Ad{\mathrm{Ad}}
\def\cI{\mathcal{I}}
\def\cX{\mathcal{X}}

\def\ba{\begin{aligned}}
\def\ea{\end{aligned}}
\def\barr{\begin{array}}
\def\earr{\end{array}}

\def\Ind{\mathrm{Ind}}

\def\gcd{\GCD}
\def\Ind{\mathrm{Ind}}
  \def\Res{\mathrm{Res}}
  \newcommand{\sgn}{\mathrm{sgn}}
  \newcommand{\One}{{1\hspace{-.14cm} 1}}

  \numberwithin{equation}{section}

\author{Fabrizio Caselli and Eric Marberg}
\title[Isomorphisms, automorphisms, and generalized involution models]{Isomorphisms, automorphisms, and generalized involution models of projective reflection groups}

\begin{document}


\maketitle

\begin{abstract} 
We investigate the generalized involution models of the projective reflection groups $G(r,p,q,n)$. This  family of groups parametrizes all quotients of the complex reflection groups $G(r,p,n)$ by scalar subgroups. 
Our classification is ultimately incomplete, but we provide several necessary and sufficient conditions for generalized involution models to exist in various cases.
In the process we solve several intermediate problems concerning the structure of projective reflection groups. We derive a simple criterion for determining whether  two groups $G(r,p,q,n)$ and $G(r,p',q',n)$ are isomorphic.  
We also describe explicitly the form of all automorphisms of $G(r,p,q,n)$, outside a finite list of exceptional cases.
Building on prior work, this allows us to prove that  $G(r,p,1,n)$ has a generalized involution model if and only if $G(r,p,1,n) \cong G(r,1,p,n)$. 
We also  classify which groups $G(r,p,q,n)$ have generalized involution models when $n=2$, or $q$ is odd, or $n$ is odd.  
\end{abstract}

\tableofcontents

\section{Introduction}

A \emph{model} for a finite group $G$ is a set $\{ \lambda_i : H_i \to \CC\}$ of linear characters  of subgroups of $G$, such that the sum of induced characters
$ \sum_i \Ind_{H_i}^G(\lambda_i)$ is equal to the multiplicity-free sum of all irreducible characters $\sum_{\psi \in \Irr(G)} \psi$. 
Models are interesting because they lead to interesting representations in which  the irreducible representations of $G$ live. 
This is especially the case when the subgroups $H_i$ are taken to be the stabilizers of the orbits of some natural $G$-action.


\begin{exa}\label{exa2}
Let $G=G(r,n)$ be the group of complex $n\times n$ matrices with exactly one nonzero entry, given by an $r$th root of unity, in each row and column. Assume $r$ is odd. Then $G$ acts on its symmetric elements by $g : X \mapsto gXg^T$, and the distinct orbits of this action are represented by the block diagonal matrices of the form 
\[ X_i \eqdef \(\barr{ll} J_{2i} & 0  \\ 0 & I_{n-2i} \earr\),\] where $J_n$ denotes the $n\times n$ matrix with ones on the anti-diagonal and zeros elsewhere. 
Write $H_i$ for the stabilizer of $X_i$ in $G$. 
The elements of $H_i$ preserve the standard copy of $\CC^{2i} $ in $ \CC^n$, inducing a map $\pi_i : H_i \to \mathrm{GL}_{2i}(\CC)$. 
If 
$\lambda_i \eqdef \mathrm{det} \circ \pi_i$ then 
 $\{ \lambda_i : H_i \to \CC \}$ is a model for $G(r,n)$ \cite[Theorem 1.2]{APR}. 
\end{exa}

The following definition  of Bump and Ginzburg \cite{BG} captures the salient features of this example.
Let $\nu$ be an automorphism of $G$ with $\nu^2=1$. Then $G$ acts on the 
set of \emph{generalized involutions} 
\[\cI_{G,\nu} \eqdef \{ \omega \in G : \omega^{-1} = \nu(\omega)\}\] by the twisted conjugation $g: \omega \mapsto g\cdot \omega \cdot \nu(g)^{-1}$.  We write
\[ C_{G,\nu}(\omega) \eqdef \{ g\in G : g\cdot \omega \cdot \nu(g)^{-1} = \omega \}\]
to denote the stabilizer of $\omega \in \cI_{G,\nu}$ under this action, and 
 say that a model $\{ \lambda_i : H_i \to \CC\}$ is a  \emph{generalized involution model} (or \emph{GIM} for short) with respect to $\nu$ if each $H_i$ is the stabilizer  $C_{G,\nu}(\omega)$ of a generalized involution $\omega \in \cI_{G,\nu}$, with each twisted conjugacy class in $\cI_{G,\nu}$ contributing exactly one subgroup. The model in the example is a GIM with respect to the inverse transpose automorphism of $G(r,n)$. 

In \cite{Ma1,Ma}, the second author classified which finite complex reflection groups have GIMs. 
Subsequently, the first author discovered an interesting reformulation 
of this classification, which suggests that these results are most naturally interpreted in the broader context of \emph{projective reflection groups}.
These groups were introduced in \cite{Ca1} and  studied, for example, in \cite{BC3}. They include as an important special case an infinite series of groups $G(r,p,q,n)$ defined as follows. 

For positive integers $r,p,n$ 
with $p$ dividing $r$, let $G(r,p,n)$ denote the subgroup 
of $G(r,n)$ consisting of the matrices  whose nonzero entries, multiplied together, form an $(r/p)$th root of unity. 
Apart from thirty-four exceptions, the irreducible finite complex reflection groups are all groups $G(r,p,n)$ of this kind. The projective reflection group $G(r,p,q,n)$ is defined as  the quotient 
\[ G(r,p,q,n) \eqdef G(r,p,n) / C_q\] where $C_q$ is the cyclic subgroup of scalar $n\times n$ matrices of order $q$. Note that for this quotient to be well-defined we must have $C_q \subset G(r,p,n)$, which occurs precisely when $q$ divides $r$ and $pq$ divides $rn$. 
 Observe also that 
$G(r,n) = G(r,1,n)$ and $G(r,p,n) = G(r,p,1,n)$.

There is an interesting notion of duality for  projective reflection groups; by definition, the projective reflection group \emph{dual} to $G=G(r,p,q,n)$ is $G^* \eqdef G(r,q,p,n)$. 
 The starting point of the present collaboration is now the following theorem which reformulates the main result of \cite{Ma}.
 
\begin{thm}\label{thm1} The complex reflection group $G=G(r,p,1,n)$ has a GIM if and only if
 $G \cong G^*$; i.e., if and only if $G(r,p,1,n) \cong G(r,1,p,n)$.
\end{thm}

\begin{remark} Explicitly, $G$ has a GIM if and only if (i)  $n\neq 2$ and $\gcd(p,n) = 1$ or (ii) $n=2$ and either $p$ or $r/p$ is odd; this is the statement of \cite[Theorem 5.2]{Ma}.
\end{remark}

Deducing this theorem from \cite[Theorem 5.2]{Ma} is straightforward, given our next main result.
Let $r,n$ be positive integers and let $p,p',q,q'$ be positive divisors of $r$ such that $pq=p'q'$  divides $rn$.
The following result simplifies and extends \cite[Theorem 4.4]{Ca1}; its proof occupies Sections  \ref{iso1-sect}, \ref{iso2-sect} and \ref{iso3-sect}.

\begin{thm}\label{thm2} The projective reflection groups $G(r,p,q,n)$ and $G(r,p',q',n)$ are isomorphic if and only if either 
(i) $ \gcd(p,n) = \gcd(p',n)$ and $\gcd(q,n) = \gcd(q',n)$
or
 %
(ii) $n=2$ and the numbers $p+p'$ and $q+q'$ and $\frac{r}{pq}$ are all odd integers.

\end{thm}

As a corollary, we can say precisely when the group $G(r,p,q,n)$ is ``self-dual'' as in Theorem \ref{thm1}.

\begin{cor}\label{cor2} The projective reflection group $G=G(r,p,q,n)$ is isomorphic to its dual $G^*=G(r,q,p,n)$ if and only if either (i) $\gcd(p,n) = \gcd(q,n)$ or (ii) $n=2$ and 
$\frac{r}{pq}$ is an odd integer.
\end{cor}

On seeing Theorem \ref{thm1} one naturally asks whether for arbitrary projective reflection groups the property of having a GIM is equivalent to self-duality. Theorem \ref{thm2} allows us to attack this question directly; its answer turns out to be false, and the rest of our results are devoted to clarifying which groups $G(r,p,q,n)$ have GIMs.
The following theorem, proved in Section \ref{n=2-sect} below, completely solves this problem in the often pathological case $n=2$.

\begin{thm}\label{n=2-thm} 
The projective reflection group
$G(r,p,q,2)$  has a GIM if and only if $(r,p,q) = (4,1,2)$ or
$G(r,p,q,2) \cong G(r,q,p,2)$.
\end{thm}

\begin{remark} By Theorem \ref{thm2}, the  condition $G(r,p,q,2) \cong G(r,q,p,2)$ holds if and only if  (i) $p$ and $q$ have the same parity or (ii) 
 $\frac{r}{pq}$ is an odd integer.
\end{remark}

A few notable differences between complex reflection groups and projective reflection groups complicates  the task of determining the existence of GIMs, and in the case $n\neq 2$ our classification is incomplete.
For example, the groups $G(r,p,q,n)$  occasionally can have conjugacy class-preserving outer automorphisms (see Proposition \ref{no-class-prop} below). The fact that the groups $G(r,p,n)$  never have such automorphisms   \cite[Proposition 3.1]{MM} was the source of a significant reduction in the proof of \cite[Theorem 5.1]{Ma} which is no longer available in many cases of interest.
Nevertheless, by carrying out a detailed analysis of the conjugacy classes and automorphisms of $G(r,p,q,n)$, we are able to prove the following theorem.

\begin{thm}\label{thm3} 
Let $G=G(r,p,q,n)$ and assume $n \neq 2$.
\begin{enumerate}
\item[(1)] If $\gcd(p,n) = 1$ then $G$ has a GIM if $q$ or $n$ is odd.
\item[(2)] If $\gcd(p,n) = 2$ then $G$ has a GIM only if $q$ is even.
\item[(3)] If $\gcd(p,n) = 3$ then $G$ has a GIM if and only if $(r,p,q,n)$ is 
    \[ (3,3,3,3) \text{ or } (6,3,3,3) \text{ or } (6,6,3,3) \text{ or }(6,3,6,3).\]

\item[(4)] If $\gcd(p,n) = 4$ then $G$ has a GIM only if 
$r \equiv p \equiv q \equiv n \equiv 4 \modu 8).$

\item[(5)] If $\gcd(p,n) \geq 5$ then $G$ does not have a GIM.

\end{enumerate}
\end{thm}

In arriving at this result, we prove a useful criterion for determining conjugacy in $G(r,p,n)$ and give an explicit description of  the automorphism group of $G(r,p,q,n)$; see Proposition \ref{grpnconjugacy} and Theorem \ref{vN=N} below.
Parts (1) and (2) of this theorem are proved as Corollary \ref{part1-cor} and Proposition \ref{part2-prop}, while parts (3)-(5) comprise Theorem \ref{main-thm}. 
We note as a corollary that the theorem provides  a complete classification when $q$ or $n$ is odd.  
This shows that projective reflection groups which are not self-dual may still possess GIMs.

\begin{cor} Let $G=G(r,p,q,n)$ and assume $n \neq 2$ and $(r,p,q,n)$ is not one of the four exceptions $(3,3,3,3)$ or $(6,3,3,3)$ or $(6,6,3,3)$ or $(6,3,6,3)$.
 If $q$ or $n$ is odd, then $G$ has a GIM if and only if $\gcd(p,n) = 1$.
\end{cor}

Combining Theorems \ref{thm2} and \ref{thm3} shows that to completely determine which projective reflection groups $G(r,p,q,n)$ have GIMs, it remains only to consider groups of the form
\[ G(2r,1,2q,2n)\quad\text{or}\quad G(2r,2,2q,2n)\quad\text{or}\quad G(8r+4,4,8q+4,8n+4).\] (Of course we only need to consider the first two types when $2n>2$). We state some conjectures concerning which of these groups should have GIMs at the end of Section \ref{last-sect}.

This research continues a line of inquiry taken up by a number of people in the past few decades. Researchers originally considered \emph{involution models}, which are simply GIMs defined with respect to the identity automorphism.
Inglis, Richardson, and Saxl described an elegant involution model for the symmetric group in \cite{IRS} (which is precisely the model in  Example \ref{exa2} when $r=1$). In his doctoral thesis, Baddeley \cite{B91} classified which finite Weyl groups have involution models. Vinroot \cite{V} extended this classification to show that the finite Coxeter groups with involution models are 
precisely those of type $A_n$, $BC_n$, $D_{2n+1}$, $F_4$,   $H_3$, and $I_2(m)$. In extending this classification to reflection groups, it is natural to consider generalized involution models, since only groups whose representations are all realizable over the real numbers can possess involution models. Adin, Postnikov, and Roichman \cite{APR} constructed a GIM for $G(r,n)$ extending Inglis, Richardson, and Saxl's original model for $S_n$, which provides the starting point of \cite{Ma1,Ma}.

As mentioned at the outset,
these sorts of classifications are interesting because they lead to interesting representations. 
We close this introduction with some recent evidence of this phenomenon. The model in Example \ref{exa2} with $r=1$ gives rise via induction to a representation of $S_n$ on the vector space spanned by its involutions. 
This representation turns out to have a simple combinatorial definition  \cite[\S1.1]{APR2007}, which surprisingly makes sense \emph{mutatis mutandis} for any Coxeter group. The generic Coxeter group representation we get in this way  corresponds to an involution  model (in the finite cases)  
 in precisely types $A_n$, $H_3$, and $I_2(2m+1)$. What's more, recent work of Lusztig and Vogan \cite{LV1,LV2,LV3} indicates that this representation 
is the specialization of a Hecke algebra representation which for Weyl groups is expected to have deep connections to the unitary representations of real reductive groups.

\section{Preliminaries}\label{prelim-sect}

Throughout we let 
$ [n] \eqdef \{ i \in \ZZ : 1 \leq i \leq n \}$  denote the set of the first $n$ positive integers. 
Fix positive integers $r$ and $n$. We write $ \ZZ_r$ 
to denote the cyclic group of order $r$; for convenience we  view this as the set $\{0,1,\dots,r-1\}$, with addition computed modulo $r$.
Likewise we write 
$ S_n$
 to denote the {symmetric group} of permutations of the set $[n]$.

Recall the definition of the group $G(r,n)$ from Example \ref{exa2} in the introduction.
This group is isomorphic to the semidirect product of $(\ZZ_r)^n$ by $S_n$ with respect to the natural action of the symmetric group and we frequently employ the following notation to refer to its elements.

\begin{defn}
Given $\pi \in S_n$ and $x=(x_1,x_2,\dots,x_n) \in (\ZZ_r)^n$,  let 
\[ (\pi,x) \in G(r,n) \]
denote the matrix whose $i$th column has $(\zeta_r)^{x_i}$  in row $\pi(i)$ and zeros in all other rows, where $\zeta_r \eqdef \exp\(\frac{2\pi  \sqrt{-1}}{r}\) $ is a fixed primitive $r$th root of unity.

\end{defn}

\begin{remark} 
When describing  elements of $(\ZZ_r)^n$, we often write $e_1,e_2\dots ,e_n$ for the standard basis of the free $\ZZ$-module $(\ZZ_r)^n$, so that 
we may then express the element $x=(x_1,x_2,\dots,x_n) \in (\ZZ_r)^n$ as $x = \sum_{i=1}^n x_i e_i$.
\end{remark}


The product of two elements $(\pi,x), (\sigma,y) \in G(r,n)$ is described as follows.
The symmetric group $S_n$ acts on $(\ZZ_r)^n$ by permuting coordinates; denote this action by letting
\[ \pi(x) \eqdef \( x_{\pi^{-1}(1)}, x_{\pi^{-1}(2)},\dots,x_{\pi^{-1}(n)}\)\qquad\text{for $\pi \in S_n$ and $x \in (\ZZ_r)^n$}.\]
One then checks that if  $\pi,\sigma \in S_n$ and $x,y \in (\ZZ_r)^n$  then 
\[ (\pi,x)(\sigma,y) = ( \pi \sigma, \sigma^{-1}(x) + y).\]
We may thus identify $S_n$ and $(\ZZ_r)^n$  as the respective subgroups of $G(r,n)$ consisting of all elements $(\pi,x)$ with $x = 0$ and $\pi =1$. 

To extract the pair $(\pi,x)$ from an arbitrary element $g \in G(r,n)$, 
we make the following definition.

\begin{defn}
Given $g = (\pi,x) \in G(r,n)$ and an integer $i \in [n]$, let
\[  |g| \eqdef \pi \in S_n \qquad\text{and}\qquad z_i(g) \eqdef x_i \in \ZZ_r
\qquad\text{and}\qquad
\Delta(g) \eqdef \sum_{i=1}^n z_i(g)
.\]
The map $g\mapsto |g|$ affords a homomorphism $G(r,n) \to S_n$, while $g \mapsto \Delta(g)$ affords a homomorphism $ G(r,n) \to \ZZ_r$.
\end{defn}

If $p$ is a positive divisor of $r$ then the subgroup $G(r,p,n)$ consists of all elements $g \in G(r,n)$ with $\Delta(g) \in 
p\ZZ_r \eqdef \{ p,2p,3p,\dots,rp\} \subset \ZZ_r$. In particular $G(r,1,n)= G(r,n)$ while $G(r,r,n) = \ker( \Delta)$.
Throughout, we write $c$ to denote the $n \times n$ scalar matrix $\zeta_r I_n$; this is the central element 
\[c = (1,e_1+e_2+\dots +e_n) \in G(r,n).\]
If $q$ divides $r$ and $pq$ divides $rn$, then $G(r,p,n)$ contains the cyclic central subgroup $C_q = \langle c^{r/q}\rangle $ of order $q$. In this case  $G(r,p,q,n)$ is defined as the quotient 
\[ G(r,p,q,n) \eqdef  G(r,p,n)/C_q.\]
We generally refer to elements of $G(r,p,q,n)$ by the same notation $(\pi,x)$ as for elements $G(r,p,n)$, with the added stipulation that $(\pi,x) =c^{ir/q}\cdot (\pi,x) $ for all $i$. This convention, while slightly abusive, does not present much ambiguity in practice.

%
%
%
%
%
%
%
%
 We define $|g|$ for $g \in G(r,p,q,n)$ exactly as for $g \in G(r,p,n)$, but the notation $z_i(g)$ is generally no longer well-defined. 
We also write
 \[ N(r,p,q,n) \eqdef  \{ g \in G(r,p,q,n) : |g| = 1\} \]
 for the normal abelian subgroup of $G(r,p,q,n)$ given by the images of the diagonal matrices in $G(r,p,n)$.

One final piece of notation which we introduce here is the \emph{inverse transpose} or \emph{complex conjugation automorphism}
\[\tau \in \Aut(G(r,p,q,n)).\] Explicitly, we define this automorphism by the formula 
\[ \tau(\pi,x) = (\pi,-x),\qquad\text{for }(\pi,x) \in G(r,p,q,n).\]
In words, note that taking the inverse of the transpose of  $g \in G(r,p,n)$ has the same effect as replacing all entries of the matrix $g$ by their complex conjugates. If we let $\tau$ denote the automorphism of $G(r,p,n)$ afforded by this operation, then $\tau$ preserves the normal subgroup $C_q$, and so descends to an automorphism of $G(r,p,q,n)$ which we denote by the same symbol. 
Note that $\tau^2 = 1$.

Related to this automorphism is the following  fact, proved by the first author in \cite{Ca2}. Let $\Irr(G)$ denote the set of complex irreducible characters of a finite group $G$, and fix positive integers $r,p,q,n$ with $p$ and $q$ dividing $r$ and $pq$ dividing $rn$. 
\begin{thm}[Proposition 4.4, Theorem 4.5, and Proposition 4.6 in \cite{Ca2}]\label{f}  Let $\tau$ denote the inverse transpose automorphism of $G=G(r,p,q,n)$ defined above. Then
\[|\cI_{G,\tau}| \eqdef \left|\left\{ \omega \in G : \omega^{-1} = \tau(\omega)  \right\} \right| \leq \sum_{\psi \in \Irr(G)} \psi(1). \] 
Furthermore, equality holds if and only if (i) $\gcd(p,n) \leq 2$ or (ii) $\gcd(p,n) = 4$ and $r\equiv p \equiv q \equiv n \equiv 4 \modu 8)$.
\end{thm}

A \emph{class-preserving outer automorphism} of a group $G$ is an automorphism which sends every element to a conjugate element but which is not a map of the form $g\mapsto xgx^{-1}$ for some $x \in G$.
The following lemma is immediate from the preceding theorem and \cite[Lemma 5.1]{Ma}.

\begin{lem}\label{5.1sub} Suppose $G=G(r,p,q,n)$ has no class-preserving outer automorphisms and equality holds in Theorem \ref{f}. Then $G$ has a GIM if and only if $G$ has a GIM with respect to the automorphism $\tau : (\pi,x) \mapsto (\pi,-x)$.
\end{lem}

\section{Isomorphism classes of $G(r,p,q,n)$} 
\label{iso1-sect}

Let $r$ and $n$ be positive integers and let $p$, $p'$, $q$, $q'$ be positive integer divisors 
of $r$. Throughout, we assume  $pq = p'q'$ and that this product divides $rn$, and we let
\[ G=G(r,p,q,n)
\qquad\text{and}\qquad
G'=G(r,p',q',n).\] The main goal of this section is   to determine a necessary and sufficient condition for $G$ and $G'$ to be isomorphic when $n \neq 2$. We will deal with the case $n=2$ in the next section. To begin, we recall from \cite[Proposition 4.2]{Ca1} the following result.
\begin{prop}\label{suff}
If $\GCD (\frac{rn}{q},p')=\GCD(\frac{rn}{q'},p)$ then for every $g\in G$ there exists a unique $g'\in G'$ such that $g$ and $g'$ have  common representatives in $G(r,n)$, and in this case the map $g\mapsto g'$ determines an isomorphism $G\cong G'$.

\end{prop}

Results in \cite[\S 4]{Ca1} completely characterize when $G \cong G^*$ if $n\neq 2$ (where we define $G^* = G(r,q,p,n)$). Our
 strategy is to generalize the ideas in that work to the present context. 

Say that a prime integer $\pi$ appears in a number $k$ with multiplicity $e$ if $\pi^e$ divides $k$ and  $\pi^{e+1}$ does not divide $k$.
A prime  is then \emph{special} if it appears in $p$ and $p'$ with different multiplicities. Since $pq = p'q'$, a prime is special if and only if it also appears in $q$ and $q'$ with different multiplicities. 
We now have the following proposition.

\begin{prop}\label{dec}
Assume that 
\[ \GCD(p,n) = \GCD(p',n) \qquad\text{and}\qquad
\GCD(q,n) = \GCD(q',n),\]
and write $\frac{rn}{pq} = \eta \delta$ where $\eta$ (respectively, $\delta$) is a positive integer equal to a product of non-special (respectively, special) primes.
Then $ G(r,\delta p,q,n)$ is well-defined and $$\qc\cong G(r,\delta p,q,n)\times \ZZ_{\delta}.$$
\end{prop}
\begin{proof}
Since $\GCD(q,n) = \GCD(q',n)$, the multiplicity of any special prime in $n$ is not  greater than the corresponding multiplicity in $q$. As $n$ divides $\eta \delta q = \frac{rn}{p}$, it follows that  $n$ divides $\eta q = \frac{rn}{\delta p}$.
Thus $\delta p$ divides $r$, and 
since $\delta pq$ likewise divides $rn$ as $\frac{rn}{\delta pq} = \eta$, 
we conclude that $G(r,\delta p,q,n)$ is well-defined.
 
 A symmetric argument using the assumption that 
 $\GCD(p,n) = \GCD(p',n)$
 shows that  $\delta q$ likewise divides $r$.  Therefore $c^{\frac{r}{\delta q}}$ is a well-defined element of order $\delta$ in $G$; let $C_\delta \cong \ZZ_\delta$ be the cyclic subgroup it generates.  Both $G(r,\delta p,q,n)$ and $C_\delta$ are normal subgroups of $\qc$, so 
to complete the proof of the proposition, we  have only to show that $G(r,\delta p,q,n)$ and $C_{\delta}$ intersect trivially. For this, it suffices to verify that 
\[(c^{\frac{r}{\delta q}})^k\in
G(r,\delta p,q,n)\quad\text{iff}\quad \tfrac{rnk}{\delta q}  \equiv 0 \modu \delta p)\quad \text{iff}\quad k\equiv 0  \modu \delta).\] 
The first equivalence follows by definition, and the second equivalence follows from the fact that if $\frac{rnk}{\delta q} = \delta p k'$ for some integers $k$, $k'$, then by dividing both sides by $p$ one obtains $\eta k = \delta k'$, which can only hold if $k$ is a multiple of $\delta $ as $\eta$ and $\delta$ are  necessarily coprime.
\end{proof}

The next pair of results establish Theorem \ref{thm2} in the case $n\neq 2$. This generalizes \cite[Theoroem 4.4]{Ca1}.

\begin{thm}\label{isomo} If $pq=p'q'$, then the groups $\qc$ and $G(r,p',q',n)$ are  isomorphic whenever
 $\GCD(p,n) = \GCD(p',n)$ and $\GCD(q,n) = \GCD(q',n)$.

\end{thm}
\begin{proof}
Write $\frac{rn}{pq} = \eta \delta$ as in Proposition \ref{dec}.
The theorem will follow immediately from Proposition \ref{dec} once we show that $G(r,\delta p,q,n)\cong G(r,\delta p',q',n)$. 
Since $\frac{rn}{q} = \eta \delta p$ and $\frac{rn}{q'} = \eta \delta p'$, it suffices by Proposition \ref{suff}  to verify that
$\GCD(\eta\delta p,\delta p') = \GCD (\eta\delta p',\delta p),$
which is equivalent to the identity
$
\GCD(\eta p,p') = \GCD(\eta p',p).$ This holds because every prime dividing $\eta$ appears in $p$ and $p'$ with equal multiplicity, and so we have in fact that $
\GCD(\eta p,p') = \GCD(\eta p',p) = \GCD(p,p')$.
\end{proof}

The next proposition implies  the converse of Theorem \ref{isomo}, provided $n\neq 2$.
\begin{prop}
   Assume  $n\neq 2$ and let $G=\qc$.
   \begin{enumerate}
   \item[(1)] The center of $G$ has order $\frac{r}{pq} \cdot \GCD(p,n)$.
   \item[(2)] The abelianization $ G/[G,G]$ of $G$ has order $ \frac{2r }{pq} \cdot \GCD(q,n)$. 
   \end{enumerate}
\end{prop}
\begin{proof}
   One can easily check that, since $n\neq 2$, the center of $G$ is given by the set of its scalar elements  (i.e. of the form $c^i$). 
   The number of scalar elements in $G$ is $\frac{1}{q}$ times the number of scalar elements in $G(r,p,n)$, which is $\frac{r}{p} \cdot \GCD(p,n)$ by \cite[Corollary 4.1]{Ma}.
   
  To prove (2), it suffices to count the linear characters of $G$ since these are equal in number to the order of $G/[G,G]$.
By \cite[\S 6]{Ca1}, the linear characters of $G(r,n)$ are parametrized by $r$-tuples of partitions $(\lambda_0,\ldots,\lambda_{r-1})$ where all partitions $\lambda_i$ are empty except one which can be either $(n)$ or $(1^n)$. The linear representations of $G(r,1,q,n)$ are parametrized by these $r$-tuples of partitions where, if the only non-empty partition appears in a position $i$, then $ni\equiv 0\modu q)$ (i.e. $(\lambda_0,\ldots,\lambda_{r-1})\in \Fer(r,q,1,n)$ in the notation of \cite[\S 6]{Ca1}). Therefore the number of linear characters of $G(r,1,q,n)$ is $\frac{2r}{q}\cdot \GCD(q,n)$. One can likewise check that, since $n\neq 2$, each linear character of $G$ is given by the common restriction of exactly $p$ distinct linear characters of $G(r,1,q,n)$.  Thus the number of linear characters of $G$ is $\frac{1}{p}$ times the number of linear characters of $G(r,1,q,n)$.
\end{proof}

Combining the preceding theorem and proposition gives this corollary, which forms one half of Theorem \ref{thm2} in the introduction.
\begin{cor}\label{cor3.5}
   Assume $n\neq 2$ and $pq=p'q'$. Then $G(r,p,q,n)\cong G(r,p',q',n)$ if and only if $\GCD(p,n) = \GCD(p',n)$ and $\GCD(q,n) = \GCD(q',n)$.
\end{cor}

\section{Isomorphism classes in rank two} 
\label{iso2-sect} 
In  this section we fix $n=2$, and assume that $p$, $p'$, $q$, $q'$ divide $r$ and $pq=p'q'$ divides $2r$. We now determine when the two groups $G=G(r,p,q,2)$ and $G'=G(r,p',q',2)$ are isomorphic. 

In referring to elements of these groups, it is convenient to abbreviate our notation  by writing $(\pi;a,b)$ for the element otherwise denoted $(\pi,(a,b)) \in G(r,p,q,2)$.
We thus view $G(r,p,q,2)$ as the set of   triples $(\pi;a,b) \in S_2 \times \ZZ_r \times \ZZ_r$ with $a+b$ divisible by $p$, where $(\pi;a,b) = (\pi';a',b')$ if and only if $\pi=\pi'$ and $a-a'\equiv b-b' \equiv k\frac{r}{q} \modu r)$ for some integer $k$. Multiplication is given by 
\[ (\pi;a,b)(\pi';a',b') = \begin{cases} (\pi\pi';a+a',b+b'),&\text{if }\pi' = 1 \in S_2,\\ (\pi\pi';b+a',a+b'),&\text{if }\pi' \neq 1 \in S_2.\end{cases}\]
%
%
We  begin with this lemma:

\begin{lem}\label{npnq}
   If $p+p'$ and $q+q'$ are both odd and $\frac{r}{pq}$ is even then $G\not \cong G'$.
\end{lem}
\begin{proof}
Since $pq=p'q'$ we may assume without loss of generality that $p'$ and $q$ are odd and that $p$ and $q'$ are even. 
By Theorem \ref{isomo} we then have that $G\cong G(r,pq,1,2)$ and $G'\cong G(r,1,p'q',2)$, and so it is enough to show that  if  $p$ and $\frac{r}{p}$ are both even then $G(r,p,1,2) \not\cong G(r,1,p,2)$.

To this end, let  $A = \{g^{r/p}:g\in G(r,p,1,2)\}$ and  $B = \{g^{r/p}:g\in G(r,1,p,2)\}$.  It suffices to show that $|A|=p$ and $|B|=p+1$. 
 It is easy to check that $A$ consists  of the distinct elements $(1;\frac{ir}{p},-\frac{ir}{p}) \in G(r,p,1,2)$ for $i \in [p]$. It is likewise a straightforward exercise to show that $B$ consists of the distinct images in $G(r,1,p,2)$ of $(1;0,\frac{ir}{p}) \in G(r,2)$ for $i \in [p]$ together with $(1;\frac{r}{2p},\frac{r}{2p}) \in G(r,2)$.
\end{proof}

Our next lemma is similar.

\begin{lem}\label{1odd} If exactly one of the four parameters $p$, $p'$, $q$, $q'$ is odd then $G\not \cong G'$.
   
\end{lem}
\begin{proof} We may assume that the unique odd parameter is either $q'$ or $p'$.
By Theorem \ref{isomo},
   if $q'$ is the unique odd parameter then   $G'\cong G(r,pq,1,2)$,
   and if $p'$ is the unique odd parameter then $G'\cong G(r,1,pq,2)$, and in either case $G \cong G(r,\frac{pq}{2},2,2)$.  It thus suffices to show that if $p$ and $\frac{r}{p}$ are even then $G(r,2p,1,2)\not\cong G(r,p,2,2)$ and $G(r,1,2p,2)\not\cong G(r,p,2,2)$.
  With these hypotheses on $p$ and $\frac{r}{p}$,  let 
  \[\begin{aligned} A &= \{ g^{r/p} : g \in G(r,2p,1,2)\}, \\ B &= \{ g^{r/p} : g \in G(r,1,2p,2)\},\\ C &= \{g^{r/p} : g \in G(r,p,2,2)\}.
  \end{aligned}\]
   As in the proof of Lemma \ref{npnq},
   it is not difficult to check that $A$ consists of the distinct elements $(1;\frac{ir}{p},-\frac{ir}{p}) \in G(r,2p,2)$ for $i \in [p]$. On the other hand,    one finds similarly that 
   $B$ consists of the distinct images in $G(r,1,2p,2)$ of the elements $(1;0,\frac{ir}{p}) \in G(r,2)$ for $i \in [p]$. Finally, $C$ consists of the distinct images in $G(r,p,2,2)$ of the elements $(1;\frac{ir}{p},-\frac{ir}{p}) \in G(r,p,2)$ for $i \in [\frac{p}{2}]$.
 Thus $|A| = |B| = p$ and $|C| = \frac{p}{2}$, which  establishes the desired non-isomorphisms.
   \end{proof}

We now examine a particular class of groups $G=G(r,p,2)$ where we can explicitly describe an isomorphism $\phi:G\rightarrow G^*$. 
\begin{lem} \label{coprime}
If $p$ or $\frac{r}{p}$ is odd then $G(r,p,1,2) \cong G(r,1,p,2)$.
\end{lem}
\begin{proof}
If $p$ is odd then $G(r,p,1,2) \cong G(r,1,p,2)$ by Theorem \ref{isomo}, so assume that $\frac{r}{p}$ is odd.
Let $p'$ be the largest power of 2 dividing $p$ (and hence also $r$), and let $q=1$ and $q' = p/p'$.
With respect to these choices of $p$, $p'$, $q$, $q'$, the special primes are precisely the odd primes  dividing $p$.  Write $\frac{2r}{p} = \frac{rn}{pq} = \eta \delta$ as in Proposition \ref{dec}, so that $\eta$ is a product of non-special primes and $\delta$ is a product of special primes, and we have
\[G(r,p,1,2) \cong G(r,\delta p,1,2) \times \ZZ_\delta
\]
and 
\[  G(r,1,p,2) \cong G(r,\delta,p,2) \times \ZZ_\delta
\cong G(r,1,\delta p,2) \times \ZZ_\delta,
\] 
the second congruence on the right following from Theorem \ref{isomo} as $\delta$ is odd.
Because $\frac{r}{p}$  is also odd,   $\eta$ is even and $\frac{\eta}{2} = \frac{r}{\delta p}$ is odd; thus $\frac{r}{\delta p}$ is a product of odd primes not dividing $p$, and so is coprime to both $p$ and $\delta$ and in particular to $\delta p$.

Since $G(r,p,1,2)\cong G(r,1,p,2)$ if $G(r,\delta p,1,2) \cong G(r,1,\delta p,2)$,
the preceding argument shows that we may assume without loss of generality that   $\frac{r}{p}$ and $p$ are coprime.
One checks that for $d = r/p'$  the map
\[\begin{array}{cccc}
 \phi: & G(r,p,1,2)&\rightarrow &G(r,1,p,2)\\
 & (\pi;i,j) &\mapsto&(\pi;i,j+di)
\end{array}\]  is a well-defined group homomorphism. 
To show that $\phi$ is an isomorphism it is  enough to demonstrate injectivity, so let $g\in G(r,p,1,2)$  such that $\phi(g)=1$. Then $g$ is necessarily of the form $(1;i,j)$ with 
\[ i+j \equiv 0 \modu p)\qquad\text{and}\qquad i\equiv j+di \equiv k\tfrac{r}{p}\modu r)\text{ for some $k \in [p]$},\] the second congruence following from the assumption that $\phi(1;i,j)=(1;i,j+di)$ represents the identity in $G(r,1,p,2)$.
These two congruences imply that $k\frac{r}{p}(2-d)$ is a multiple of $p$.
Since $d$ is odd,  no number dividing $2-d$ divides either 2 or $d$, and as every odd prime dividing $p$ also divides $d$, it follows that $\GCD(2-d,p)=1$.
Since $\frac{r}{p}$ is  coprime to $p$ by hypothesis, we conclude that $k$ is a multiple of $p$, which implies that  
 $i\equiv  j \equiv 0 \modu r)$ and in turn that $g=1$, as desired.
\end{proof}

Gathering together the preceding results yields the following summary theorem.

\begin{thm}\label{n2-thm} Assume $pq=p'q'$. Then
   $G(r,p,q,2)\cong G(r,p',q',2)$  if and only if one of the following mutually exclusive conditions holds:
   \begin{enumerate}
   \item[(i)] $p+p'$ and $q+q'$ are both even;
   \item[(ii)] $p+p'$, $q+q'$, and $\frac{r}{pq}$ are all odd integers.
\end{enumerate}
\end{thm}
\begin{proof}
If the first condition holds then $G \cong G'$ by 
 Theorem \ref{isomo}.
 If the second condition holds then
 since $pq=p'q'$, exactly one of $p$, $q$ is odd and it follows that $pq$ in fact divides $r$.
In this case, we may 
   assume  that $p$ and $q'$ are even and that $p'$ and $q$ are odd. 
Theorem \ref{isomo} then implies that $G \cong G(r,pq,1,2)$ and $G'\cong G(r,1,pq,2)$, while Lemma \ref{coprime} implies that  $G(r,pq,1,2)\cong G(r,1,pq,2)$.
 
  If $p+p' $ and $q+q'$ are both odd but $\frac{r}{pq}$ is even then $G\not\cong G'$ by Lemma \ref{npnq}. If $p+p'$ and $q+q'$ have different parities then 
exactly one of the  parameters $p$, $p'$, $q$, $q'$ is odd as $pq=p'q'$, so $G\not\cong G'$ by Lemma \ref{1odd}.   
\end{proof}

Combining this theorem with Corollary \ref{cor3.5} gives Theorem \ref{thm2} in the introduction.

\section{Constructing an isomorphism explicitly}\label{iso3-sect}

We present here an alternative proof of Theorem \ref{isomo} by constructing an explicit isomorphism between $G = G(r,p,q,n)$ and $G'=G(r,p',q',n)$. It is our hope that this construction
may be useful at some point in explaining why groups with $G(r,p,q,n) \cong G(r,q,p,n)$ often tend to have generalized involution models.

Let $r$ and $n$ be any positive integers.
Fix positive divisors $p$, $p'$, $q$, $q'$
of $r$ with $pq=p'q'$ dividing $rn$, and 
for each prime $\pi$ define 
\[a_\pi,\quad a'_\pi,\quad b_\pi,\quad b'_\pi,\quad c_\pi,\quad d_\pi\] as the multiplicities of $\pi$ in the prime factorizations of $p$, $p'$, $q$, $q'$, $r$, $n$, respectively. We  first prove this technical result.

\begin{lem}\label{=><-lem}
   Assume that 
    \[
  \GCD(p,n) = \GCD(p',n)
  \qquad\text{and}\qquad
  \GCD(q,n) = \GCD(q',n).
  \] 
   Then there exists an integer $x$ such that for all primes $\pi$ dividing $rn$, the following three-part condition holds:
\begin{equation*}
\begin{cases} \begin{aligned}
\textrm{ If $a_{\pi}=a'_{\pi}$ then }\qquad&  x\equiv 0 &&\modu \pi^{a_{\pi}+1})\\
\textrm{ If $a_{\pi}>a'_{\pi}$ then }\qquad&x\equiv\pi^{a'_{\pi}-d_{\pi}}&&\modu \pi^{a'_{\pi}-d_{\pi}+1}) \\
\textrm{ If $a_{\pi}<a'_{\pi}$ then }\qquad&\tfrac{rn}{pq}x+\tfrac{r}{q}\equiv \pi^{c_{\pi}-b'_{\pi}}&&\modu \pi^{c_{\pi}-b'_{\pi}+1}).
\end{aligned}
\end{cases}
\end{equation*}

\end{lem}
\begin{proof}

   By the Chinese remainder theorem it suffices  to verify that 
   for each prime $\pi$ dividing $rn$, there exists  $x \in \ZZ$ satisfying the relevant congruence.
   If $a_\pi = a_\pi'$ then such an integer $x$ clearly exists. If $a_\pi > a'_\pi$ then $\GCD(p,n) = \GCD(p',n)$ implies $d_\pi \leq a_\pi'$, so the second congruence is   well-defined, and it too clearly has a solution.
   
  If $a_\pi < a_\pi'$ then $pq=p'q'$ implies $b_\pi' < b_\pi$ and 
  $\GCD(p,n) = \GCD(p',n)$  implies $d_\pi \leq a_\pi$ and  the fact that $q'$ divides $r$  implies $b_\pi' \leq c_\pi$. 
Thus 
  the corresponding congruence is at least well-defined. To show that it has a solution, it is enough to 
check  that $\GCD(\pi^{c_{\pi}-b'_{\pi}+1},\frac{rn}{pq})$ divides $\pi^{c_{\pi}-b'_{\pi}}-\frac{r}{q}$. 
This is equivalent to the inequality 
\[\min \{ c_\pi -b_\pi'+1 , (c_\pi +d_\pi)- (a_\pi+b_\pi) \} \leq \min \{ c_\pi - b_\pi', c_\pi - b_\pi\}.\]
Since $b_\pi' < b_\pi$ and $d_\pi \leq a_\pi$, the left-hand side is $(c_\pi - b_\pi) + (d_\pi - a_\pi)$ and the right-hand side is $c_\pi -b_\pi$, and the inequality follows.
\end{proof}

The integer $x$ given in the previous result satisfies a simpler set of congruences, which we describe in the following lemma.

\begin{lem}\label{system}
   Assume
      $\GCD(p,n) = \GCD(p',n)$
  and
   $\GCD(q,n) = \GCD(q',n)$
   and let $x$ be the integer given in Lemma \ref{=><-lem}.  Then both of the following hold:
   
   \begin{enumerate}
   \item[(1)] 
    $ \tfrac{rn}{pq} x + \tfrac{r}{q} $ and $  \tfrac{r}{p} x $ are both divisible by $\frac{r}{q'}$. 

\item[(2)]   For all primes $\pi$ dividing both $\frac{rn}{pq}$ and $\frac{r}{q}$,
we 
have
    $  \tfrac{rn}{pq\pi} x + \tfrac{r}{q\pi} \not \equiv 0 \modu \tfrac{r}{q'})$.
    \end{enumerate}

\end{lem}
\begin{proof}
To prove that (1) holds, it suffices to show   $\tfrac{rn}{pq} x + \tfrac{r}{q} \equiv   \tfrac{r}{p} x \equiv 0 \modu \pi^{c_\pi - b_\pi'})$ for all primes $\pi$ dividing $rn$.
Deriving  this set of  congruences from Lemma \ref{=><-lem} is straightforward.

Let $\pi$ be a prime dividing both $\frac{rn}{pq}$ and $\frac{r}{q}$.
To complete the lemma, it is enough to show that $\tfrac{rn}{pq\pi} x + \tfrac{r}{q\pi} \not \equiv 0 \modu \pi^{c_\pi-b_\pi'})$.   The following statements affirming this are again straightforward consequences of Lemma \ref{=><-lem}.
 First, if $a_{\pi}=a'_{\pi}$ then
 \[ \tfrac{rn}{pq\pi}x\equiv 0 \modu \pi^{c_{\pi}-b'_{\pi}}) \qquad\text{but}\qquad \tfrac{r}{q\pi}\not \equiv 0 \modu \pi^{c_{\pi}-b'_{\pi}}).\] On the other hand, if $a_{\pi}>a'_\pi$, so that $b_\pi<b'_\pi$, then  similarly 
 \[\tfrac{rn}{pq\pi}x\not \equiv 0 \modu \pi^{c_\pi-b'_\pi}) \qquad\text{but}\qquad \tfrac{r}{q\pi}\equiv 0 \modu \pi^{c_\pi-b'_\pi}).\] Finally, if $a_{\pi}<a'_\pi$, we have 
 $\tfrac{rn}{pq\pi}x+\tfrac{r}{q\pi}\equiv \pi^{c_{\pi}-b'_\pi-1} \not\equiv 0\modu \pi^{c_{\pi}-b'_\pi}).$
\end{proof}

The next theorem  describes a surjective homomorphism $G(r,p,n) \to G(r,p',q',n)$ which will descend to the promised  isomorphism $G \xrightarrow{\sim} G'$.

\begin{thm}
   Assume
      $\GCD(p,n) = \GCD(p',n)$
  and
   $\GCD(q,n) = \GCD(q',n)$
   and let $x$ be the integer given in Lemma \ref{system}. Then the map defined, with our usual slight abuse of notation, by
\[\begin{array}{cccc}
   \varphi: & G(r,p,n)&\longrightarrow &G(r,p',q',n)\\
& g&\mapsto& g\cdot c^{\frac{\Delta(g)}{p}x}
\end{array}
\]
is a surjective group homomorphism whose kernel is the cyclic central subgroup $C_q = \langle c^{r/q}\rangle$ of $G(r,p,n)$.

\end{thm}
\begin{proof}
The map $\varphi$ is well-defined as a consequence of the following observations:
\begin{itemize}

 \item The product $g\cdot  c^{\frac{\Delta(g)}{p}x}$ belongs to $G(r,p',n)$, since multiplying the congruence $\frac{rn}{pq}x+\frac{r}{q}\equiv 0 \modu \frac{r}{q'})$ given in Lemma \ref{system} by $pq=p'q'$  shows that
 $nx+p$ is divisible by $p'$, whence
$\Delta\(g\cdot  c^{\frac{\Delta(g)}{p}x}\)=\tfrac{\Delta(g)}{p}(nx+p) \equiv 0 \modu p').$

  \item The image of $g\cdot  c^{\frac{\Delta(g)}{p}x}$ in $G(r,p',q',n)$ does not depend on the representative chosen for $\Delta(g)$ modulo $r$,
 since the congruence   $\frac{r}{p} x \equiv 0 \modu \frac{r}{q'})$ given in Lemma \ref{system}.
implies  that
 $c^{\frac{r}{p} x}$ belongs to the subgroup of $G(r,p',n)$ generated by $c^{\frac{r}{q'}}$.
 
\end{itemize}
The  fact that $\Delta(g h)=\Delta(g)+\Delta(h)$ shows that $\varphi$ is a group homomorphism, and we have  $c^{\frac{r}{q}} \in \ker \varphi$ since the congruence $\frac{rn}{pq} x + \frac{r}{q} \equiv 0 \modu \frac{r}{q'})$   implies that $\varphi(c^{\frac{r}{q}})$ represents the identity in $G(r,p',q',n)$.

Thus $ \langle c^{\frac{r}{q}}\rangle \subset \ker \varphi$. To show that this inclusion is an equality, note that $\ker \varphi$ is  necessarily a  subgroup of the cyclic group of scalar elements in $G(r,p,n)$.
Hence, if $\langle c^{\frac{r}{q}}\rangle$ is a proper subgroup of $\ker \varphi$, then by basic properties of cyclic groups and their subgroups, 
 there must exist a prime $\pi$ dividing $\frac{r}{q}$ such that $c^{\frac{r}{q\pi}}\in \ker \varphi$. 
For $c^{\frac{r}{q\pi}}$ to belong to $G(r,p,n)$,  the prime $\pi$ must also divide $\frac{rn}{pq}$; in this case, however, part (2) of Lemma \ref{system}  
implies that $\varphi(c^{\frac{r}{q\pi}})=c^{\frac{r}{q\pi}+\frac{rn}{pq\pi}x}\neq 1 \in G'$. We conclude that $\ker \varphi=\langle c^{\frac{r}{q}}\rangle$. The fact that $\varphi$ is surjective then follows by cardinality reasons.
\end{proof}

The following corollary is an immediate consequence of the preceding result.

\begin{cor}
   If 
   $\GCD(p,n) = \GCD(p',n)$
  and
   $\GCD(q,n) = \GCD(q',n)$
and $x$ is the integer given in Lemma \ref{system} then the following map
is an isomorphism: \[\begin{array}{cccc}
   \varphi: & G(r,p,q,n)&\longrightarrow &G(r,p',q',n)\\
& g&\mapsto& g\cdot c^{\frac{\Delta(g)}{p}x}
\end{array}
\] 
\end{cor}

\begin{remark} Our notation here is abusive, and one should interpret our formula as meaning that $\varphi$ sends the image of $g \in G(r,p,n)$ in $G(r,p,q,n)$ to the image of the element $g\cdot   c^{\frac{\Delta(g)}{p}x}$ in $G(r,p',q',n)$.
\end{remark}

\section{Generalized involution models in rank two} 
\label{n=2-sect}

In this section we determine which of the projective reflection groups $G(r,p,q,2)$ have GIMs, proving Theorem \ref{n=2-thm} from the introduction.
Thus, fix positive integers $r,p,q$  with $p$ and $q$ dividing $r$ and $pq$ dividing $2r$. 
We represent elements of $G(r,p,q,2)$ as triples $(\pi;a,b) \in S_2 \times \ZZ_r \times \ZZ_r $ as in Section \ref{iso2-sect}. 
Here, we also let
 $\sigma$ denote the nontrivial permutation 
 \[\sigma=(1,2) \in S_2.\]
 %
To begin, from our results so far we have this lemma:


\begin{lem}\label{q1}
If $q$ or $r/q$ is odd then $G(r,1,q,2)$ has a GIM.
\end{lem}

\begin{proof}
If $q$ or $r/q$ is odd, then
$G(r,1,q,2) \cong G(r,q,1,2)$ by Corollary \ref{cor2} while $G(r,q,1,2)$ has a GIM by  Theorem \ref{thm1}.  
\end{proof}

Define $ \tau$ as in Section \ref{prelim-sect} as the involution of $G=G(r,p,q,2)$ given by the map $(\pi;a,b) \mapsto (\pi;-a,-b)$. The corresponding set of generalized involutions
\[\cI_{G,\tau}\eqdef \{ \omega \in G : \omega^{-1} = \tau(\omega)\}\]
 consists of elements of the form $(\pi;a,b) \in G$ such that either (i) $\pi=1\in S_2$ or 
(ii) $\pi\neq 1 \in S_2$ and $a=b$ or (iii) $\pi \neq 1 \in S_2$ and $q$ is even and $a=b+r/2 \in \ZZ_r$.
The following result shows that we only need to consider this automorphism to determine if $G$ has a GIM.

\begin{lem}\label{no-class-lem} The group $G(r,p,q,2)$ has no class-preserving outer automorphisms, and so
 $G(r,p,q,2)$  has a GIM if and only if it has a GIM with respect to $\tau$. 
\end{lem}

\begin{proof}
The result holds if $q=1$ by \cite[Proposition 3.1]{MM}, so we may assume $q>1$ (so that also $r>1$).
The group $G=G(r,p,q,2)$ is generated by the three elements $s = (1;1,-1)$, $t=(1;p,0)$, and  $\sigma =((1,2);0,0)  \in S_2$. Suppose $\nu \in \Aut(G)$ is class-preserving; we argue that $\nu$ is inner. Since $G$ is a semidirect product of the abelian groups $S_2$ and $N(r,p,q,2)$,  for some integers $i,j,k,l$ we have 
\[  \nu(\sigma) = s^i t^j \sigma t^{-i} s^{-j} \qquad\text{and}\qquad \nu(s) = \sigma^k s \sigma^{-k}\qquad\text{and}\qquad \nu(t) = \sigma^l t \sigma^{-l}.\]
Let $x = s^i t^j \sigma^k \in G$ and define $\nu'$ as the automorphism $\nu' : g \mapsto x^{-1} \nu(g) x$. Then $\nu'$ is class-preserving with $\nu'(s) = s$ and $\nu'(\sigma) = \sigma$ and $\nu'(t) \in \{ t, \sigma t \sigma \}$, and to show that $\nu$ is inner it suffices to show that $\nu'$ is inner.

Certainly $\nu'$ is inner if $\nu'(t) = t$ or $p=r$ (in which case $t=1$) so suppose $\nu'(t) = \sigma t \sigma$ and $p<r$.
Let $z = st = (1;p+1,-1)$ and $z' = \sigma z \sigma=  (1;-1,p+1)$ so that $\{ z,z'\}$ comprises a conjugacy class in $G$. Then $\nu'(z) =\nu'(s) \nu'(t)= (1;1,p-1) \in \{ z,z'\}$ since $\nu'$ is class-preserving. This implies that for some integer $k$ either 
\[ p \equiv -p-2 \equiv k \tfrac{r}{q} \modu r)\qquad\text{or}\qquad   -2 \equiv 2 \equiv k\tfrac{r}{q} \modu r).\]
The first congruence implies $p \in \{ r-1, \frac{r}{2} - 1\}$, in which case since $p$ divides $r$ we must have $r \in \{2,4,6\}$,
while the second congruence implies $r \in \{2,4\}$.

In either case we must have $r \in \{ 2,4,6\}$.
We now observe that applying $\nu'$ to both sides of the identity $t \sigma t^{-1} \sigma = s^p$ gives $s^{-p} = \sigma t \sigma t^{-1} = s^p$.  This equation holds in $G$ if and only if for some integer $k$ we have \[- 2p \equiv 2p \equiv k \tfrac{r}{q} \modu r).\]  
The first part of this congruence implies $4p$ is a multiple of $r$, so since $p<r$ we must have $p \in \{r/2, r/4\}$.
 If $p=r/2$ then $q \in \{2,4\}$ since $q$ divides $2r/p$ and we assume $q>1$; in this case $r/2$ is a multiple of $r/q$ so we have $t = (1;\frac{r}{2},0) = (1;0,\frac{r}{2}) = \sigma t \sigma $ in $G$, whence $\nu' =1$ in inner. 
On the other hand, if $p=r/4$ then we must have $r=4$ and $p=1$. In this case $t$ and $\sigma$ generate $G$, so $\nu'$ must coincide with the inner automorphism $g \mapsto \sigma g \sigma$ since it does so on the generators $t$, $\sigma$.

We conclude that all class-preserving automorphism of $G$ are inner. The last part of the lemma now follows from Lemma \ref{5.1sub}.
\end{proof}

This result leads to a less trivial lemma.

\begin{lem} \label{q3} If $q$ and $r/q$ are both even and $(r,q) \neq (4,2)$, then $G(r,1,q,2)$ does not have a GIM.
\end{lem}
 
\begin{proof} Let $G = G(r,1,q,2)$. Assuming both $q$ and $r/q$ are even integers, the $\tau$-twisted centralizer of the generalized involution $\omega\eqdef (1;0,1) \in G$ is  the subgroup
\[C_{G,\tau}(\omega) = \left\{ (1;0,0), (1;\tfrac{r}{2q},\tfrac{r}{2q}), (1;0,\tfrac{r}{2}), (1;\tfrac{r}{2q},\tfrac{r}{2q}+\tfrac{r}{2}) \right\}
\cong \ZZ_2 \times \ZZ_2.\]
This subgroup has four linear characters, given by the homomorphisms $\lambda^{\epsilon,\nu} : C_{G,\tau}(\omega) \to \CC^\times$ for $\epsilon,\nu \in \{0,1\}$   with $(1;\tfrac{r}{2q},\tfrac{r}{2q}) \mapsto (-1)^{\epsilon}$ and $(1;0,\tfrac{r}{2})\mapsto (-1)^\nu$. To prove the lemma, it suffices to show that if $(r,q) \neq (4,2)$ then $\Ind_{C_{G,\tau}(\omega)}^G(\lambda^{\epsilon,\nu})$ is never multiplicity-free.

From basic character theory, all the irreducible characters of $G$ have degree 1 or 2, and the  degree 2 irreducible characters are the functions 
\[ \chi^{x,y} : (\pi;a,b) \mapsto \begin{cases} \zeta_r^{ax+by} + \zeta_r^{ay+bx},&\text{if }\pi = 1 \in S_2, \\ 0,&\text{if }\pi \neq 1 \in S_2,\end{cases}\qquad\text{for }(\pi;a,b) \in G,\]
where $\zeta_r = \exp(2\pi i /r)$ is a primitive $r^{\mathrm{th}}$ root of unity, and
where $x$ and $y$ range over all integers with 
$0\leq x <y < r$ and $x+y \equiv 0 \modu q)$.
From this formula, it is easy to see that
\[ \Res_{C_{G,\tau}(\omega)}^{G} (\chi^{x,y}) =2 \lambda^{\epsilon,\nu},
\qquad\text{where }\epsilon \equiv \tfrac{x+y}{q} \modu 2) \text{ and }\nu \equiv x \equiv y \modu 2).
\]
If $(r,q) = (4,2)$ then we must have $(x,y) \in \{(0,2),(1,3)\}$ whence  the restriction of a degree 2 irreducible character of $G$ to $C_{G,\tau}(\omega)$ is either $2\lambda^{1,0}$ or $2\lambda^{0,1}$. 

If $(r,q) \neq (4,2)$, then for any $\epsilon,\nu \in \{0,1\}$, one can find integers $x,y$ with $0\leq x <y < r$ and $x+y \equiv 0 \modu q)$ such that $\epsilon \equiv \tfrac{x+y}{q} \modu 2)$ and $\nu \equiv x \equiv y \modu 2)$.  For example, if $\epsilon = \nu = 0$ then we can take $(x,y)=(2,2q-2)$ if $q\geq 4$ or $(x,y)=(2,4q-2)$ if $q=2$ and $r\geq 8$.  By Frobenius reciprocity, this implies  that if $(r,q)\neq (4,2)$ then $\Ind_{C_{G,\tau}(\omega)}^G (\lambda^{\epsilon,\nu})$  fails to be multiplicity-free for all $\epsilon,\nu \in \{0,1\}$. Hence, by the preceding lemma, $G$ does not have a  GIM if $q$ and $r/q$ are both even and $(r,q)\neq (4,2)$. 
\end{proof}

The only group $G(r,1,q,2)$ to which Lemmas \ref{q1} and \ref{q3} do not apply is the non-abelian group of order sixteen $G(4,1,2,2)$. This group does have a GIM, though writing down and carefully checking all the data specifying this is a tedious and not very instructive exercise, which we therefore omit.
Combining this fact with Lemmas \ref{q1} and \ref{q3} gives the following proposition.

\begin{prop}\label{q-prop}
$G(r,1,q,2)$ has a GIM if and only if $(r,q) = (4,2)$ or $q$ is odd or $r/q$ is odd.
\end{prop}

Now we consider all groups of the form $G(r,p,q,2)$.
We can treat several cases at once using the preceding proposition and our knowledge of when $G(r,p,q,2) \cong G(r,p',q',2)$.

\begin{lem}\label{n2-lem} 
If $p$ and $q$ are not both even, then $G(r,p,q,2)$ has a GIM if and only if (i) $p$ and $q$ are both odd, (ii) $\frac{r}{pq}$ is an odd integer, or (iii) $(r,p,q) = (4,1,2)$.
\end{lem}

\begin{proof}
Everything follows from Theorems \ref{thm1} and  \ref{thm2} and Proposition \ref{q-prop}. If $p$ and $q$ are both odd then $G(r,p,q,2) \cong G(r,pq,1,2)$ so the group has a GIM since $pq$ is odd. 
If $p$ is even and $q$ is odd then $G(r,p,q,2) \cong G(r,pq,1,2)$, so the group has a GIM if and only if $\frac{r}{pq}$ is odd.
Likewise, if $p$ is odd and $q$ is even then $G(r,p,q,2)\cong G(r,1,pq,2)$, so the group has a GIM if and only if $\frac{r}{pq}$ is odd or $(r,pq) = (4,2)$.
\end{proof}

It remains to consider the groups  $G(r,p,q,2)$ with $p$ and $q$  both even. We will find that these always have GIMs; to prove this it suffices by Theorem \ref{thm2} to consider the groups $G(r,2,q,2)$ where $q$ is  even. 
The nonlinear irreducible characters of $G(r,2,q,2)$ are the (not generally distinct) functions 
\[\chi^{x,y} : (\pi;a,b) \mapsto \begin{cases} \zeta_r^{ax+by} + \zeta_r^{ay+bx},&\text{if }\pi = 1 \in S_2, \\ 0,&\text{if }\pi \neq 1 \in S_2,\end{cases}\]
where $x,y$  range over all integers such that $x+y \equiv 0 \modu q)$ and $x \not \equiv y  \modu \tfrac{r}{2})$. Note that we always have $\chi^{x,y}=\chi^{y,x}$ and $\chi^{x,y} = \chi^{x+\frac{r}{2}, y +\frac{r}{2}}$.
The linear characters of $G(r,2,q,2)$ are alternatively the (not generally distinct) functions 
\[\lambda^{z,\epsilon} : (\pi;a,b) \mapsto \zeta_{r}^{(a+b)z} \cdot \sgn(\pi)^\epsilon
\]
and
\[
\nu^{w,\epsilon} : (\pi;a,b) \mapsto (-1)^a\cdot \zeta_{r}^{(a+b)w} \cdot \sgn(\pi)^\epsilon,\]
where  $z$ is a multiple of $\tfrac{q}{2}$ and $\epsilon \in \{0,1\}$ and $w$ is an integer with $(\zeta_{q/2})^w = (-1)^{r/q}$.
Note that $\nu^{z,\epsilon}$ is only a well-defined character if $\frac{r}{q}$ is even and $w$ is a multiple of $\frac{q}{2}$, or if $\frac{r}{q}$ is odd and $\frac{q}{2}$ is even and $w$ is an odd multiple of $\frac{q}{4}$.
Also, observe
that $\lambda^{z,\epsilon} = \lambda^{z+\frac{r}{2},\epsilon}$ and 
$\nu^{w,\epsilon} = \nu^{w+\frac{r}{2},\epsilon}$.

Continue to let  $\tau$ denote the automorphism of $G(r,2,q,2)$ given by $(\pi;a,b) \mapsto(\pi;-a,-b)$.  We now have the following sequence of lemmas.

\begin{lem} Assume $q$ is even. If $\frac{r}{2}$ is odd then $G(r,2,q,2)$ has a GIM.
\end{lem}

\begin{proof}
Write $G=G(r,2,q,2)$. In the situation of the lemma, $\frac{r}{q}$ is odd so  $\Irr(G) = \{ \chi^{x,y} \} \cup \{ \lambda^{z,\epsilon}\}$ with $x,y,z,\epsilon$ as above. The following facts hold by routine arguments: there are only two $\tau$-twisted conjugacy classes in $G$; these are represented by the elements $(1;0,0)$ and $(\sigma;0,0)$; and the corresponding $\tau$-twisted centralizers are the subgroups $A \eqdef \langle (\sigma;0,0)\rangle \cong S_2$ and $B \eqdef 
G(r,\tfrac{2r}{q},q,2)$.

 Write $\sgn_A$ and $\One_B$ for the nonprincipal and principal irreducible characters of $A$ and $B$, respectively.  Frobenius reciprocity implies that 
$  \Ind_A^G(\sgn_A) = \sum_{z} \lambda^{z,1} + \sum_{x,y} \chi^{x,y},$ where the (multiplicity-free) sums range over all allowable values of $x$, $y$, and $z$.
Each linear character $\lambda^{z,\epsilon}$ (with $z$ a multiple of $\frac{q}{2}$) restricts to the linear character $(\pi;a,b) \mapsto \sgn(\pi)^\epsilon$ of $B$. It follows that $\lambda^{z,\epsilon}$ is a constituent of $\Ind_B^G(\One_B)$ if and only if $\epsilon = 0$, whence 
every irreducible character of $G$ is a constituent of $\Ind_A^G(\sgn_A) + \Ind_B^G(\One_B)$.  Since this sum of induced characters has the same degree as $\sum_{\psi \in \Irr(G)} \psi$ by Theorem \ref{f}, we conclude that $\{ \sgn_A, \One_B\}$ is a GIM for $G$.
\end{proof}

\begin{lem} 
Assume $q$ is even. If $\frac{r}{2}$ is even but $\frac{r}{q}$ is odd  then $G(r,2,q,2)$ has a  GIM.
\end{lem}

\begin{proof}
Again let $G=G(r,2,q,2)$ and note that $\frac{q}{2}$ is even and that  $\Irr(G) = \{ \chi^{x,y} \} \cup \{\lambda^{z,\epsilon}\} \cup \{ \nu^{w,\epsilon}\}$, with the parameters $x,y,z,w,\epsilon$ subject to the conditions  above.  In particular, $w$ must be an odd multiple of $\frac{q}{4}$.
There are four $\tau$-twisted conjugacy classes in $G$, represented by the elements $(1;0,0)$, $(1;0,2)$, $(\sigma;0,0)$, $(\sigma;0,\tfrac{r}{2})$, with corresponding $\tau$-twisted centralizers  $A$, $B$, $C$,  $C$, where 
\[ A \eqdef \Bigl\langle (\sigma;0,0),\ (1;0,\tfrac{r}{2}) \Bigr\rangle \cong S_2\times S_2
\]
and
\[
B \eqdef \Bigl\langle (\sigma;1,-1),\ (1;0,\tfrac{r}{2}) \Bigr\rangle \cong S_2\times S_2
\] and $C \eqdef G(r,\tfrac{2r}{q},q,2)$.  Choose linear characters of these subgroups as follows:
\begin{enumerate}
\item[$\bullet$] Let $\alpha=\One_A$ be the principal character of $A$.
\item[$\bullet$] Let $\beta$ be the linear character of $B$ with 
\[\beta(\sigma;1,-1) = \beta(1;0,\tfrac{r}{2}) = -1.\]
\item[$\bullet$] Let $\gamma$ be the common restriction of $\lambda^{z,1}$ (for  multiples $z$ of $\frac{q}{2}$) to $C$.
\item[$\bullet$] Let $\gamma'$ be the common restriction of $\nu^{w,1}$ (for  odd multiples $w$ of $\frac{q}{4}$) to $C$.
\end{enumerate}
\normalsize
Using Frobenius reciprocity with the explicit character formulas provided above, one can check that $\chi^{x,y}$ is a constituent of $\Ind_A^G(\alpha)$ for all $x,y$ with $x$ (hence also $y$)  even and of $\Ind_B^G(\beta)$ for all $x,y$ with $x$  odd;
that
$\lambda^{z,0}$ and $\lambda^{z,1}$ are  constituents for all $z$ of $\Ind_A^G(\alpha)$ and $\Ind_C^G(\gamma)$, respectively;  that $\nu^{w,0}$ is a constituent for all $w$ of exactly one of $\Ind_A^G(\alpha)$ or $\Ind_B^G(\beta)$ depending on the parity of $\frac{q}{4}$; and that
$\nu^{w,1}$ is a constituent of $\Ind_C^G(\gamma')$ for all $w$. As in the previous lemma, these observations suffice to show that $\{\alpha,\beta,\gamma,\gamma'\}$ is a GIM for $G$ since $\Ind_A^G(\alpha) + \Ind_B^G(\beta) + \Ind_C^G(\gamma) + \Ind_C^G(\gamma')$ has the same degree as $\sum_{\psi \in \Irr(G)} \psi$ by Theorem \ref{f}.
 \end{proof}

\begin{lem} 
Assume $q$ is even. If $\frac{r}{q}$ is even then $G(r,2,q,2)$ has a  GIM.
\end{lem}

\begin{proof}
Again let $G=G(r,2,q,2)$ and note that $\frac{r}{2}$ is even and that  $\Irr(G) = \{ \chi^{x,y} \} \cup \{\lambda^{z,\epsilon}\} \cup \{ \nu^{w,\epsilon}\}$, where now both $z$ and $w$ must be integer multiples of $\frac{q}{2}$.
There are eight $\tau$-twisted conjugacy classes in $G$, represented by the elements $(1;0,0)$, $(1;1,1)$, $(1;0,2)$, $(1;1,3)$, $(\sigma;0,0)$, $(\sigma;\frac{r}{2q}, \frac{r}{2q})$, $(\sigma;0,\frac{r}{2})$, $(\sigma;\frac{r}{2q}, \frac{r}{2q}+\frac{r}{2})$, with corresponding $\tau$-twisted centralizers  $A$, $A$, $B$, $B$, $C$,  $C$, $C$, $C$, where 
\[A \eqdef \Bigl\langle (\sigma;0,0),\ (1;0,\tfrac{r}{2}),\ (1;\tfrac{r}{2q}, \tfrac{r}{2q}) \Bigr\rangle \cong (S_2)^3
\]
and
\[
B \eqdef \Bigl\langle (\sigma;1,-1),\ (1;0,\tfrac{r}{2}),\ (1;\tfrac{r}{2q}, \tfrac{r}{2q}) \Bigr\rangle \cong (S_2)^3
\] and $C \eqdef G(r,\tfrac{r}{q},q,2)$. Choose linear characters of these subgroups as follows:
\begin{enumerate}
\item[$\bullet$] Let $\alpha=\One_A$ be the principal character of $A$.
\item[$\bullet$] Let $\alpha'$ be the linear character of $A$ with 
\[\alpha'(\sigma;0,0) = \alpha'(1;0,\tfrac{r}{2}) = 1
\qquad\text{and}\qquad \alpha'(1;\tfrac{r}{2q},\tfrac{r}{2q}) = -1.\]
\item[$\bullet$] Let $\beta$ be the linear character of $B$ with 
\[ \beta(\sigma;1,-1) = \beta(1;0,\tfrac{r}{2}) = -1\qquad\text{and}\qquad \beta(1;\tfrac{r}{2q},\tfrac{r}{2q}) =1.\]
\item[$\bullet$] Let $\beta'$ be the linear character of $B$ with 
\[\beta(\sigma;1,-1) = \beta(1;0,\tfrac{r}{2}) =\beta(1;\tfrac{r}{2q},\tfrac{r}{2q}) =-1.\]
\item[$\bullet$] Let $\gamma$ be the common restriction of $\lambda^{z,1}$ for $z \in \{kq : k \in \ZZ\}$ to $C$.
\item[$\bullet$] Let $\gamma'$ be the common  restriction of $\lambda^{z+\frac{q}{2},\epsilon}$ for $z \in \{kq : k \in \ZZ\}$ to $C$, where $\epsilon \equiv \frac{q}{2} \modu 2)$.
\item[$\bullet$] Let $\gamma''$ be the common restriction of $\nu^{w,1}$ for $w \in \{ kq : k \in \ZZ\}$ to $C$.
\item[$\bullet$] Let $\gamma'''$ be the common restriction of $\nu^{w+\frac{q}{2},1}$ for $w \in \{ kq : k \in \ZZ\}$ to $C$.

\end{enumerate}
\normalsize
The following facts are routine consequences of Frobenius reciprocity. The irreducible character
$\chi^{x,y}$ is  a constituent of $\Ind_A^G(\alpha)$, $\Ind_A^G(\alpha')$, $\Ind_B^G(\beta)$, or $\Ind_B^G(\beta')$, respectively, if the parities of $(x, \frac{x+y}{q})$ are 
(even, even), (even, odd), (odd, even), or (odd, odd).
For all $z \in \{kq : k \in \ZZ\}$, the irreducible characters 
$\lambda^{z,0}$, $\lambda^{z+\frac{q}{2},1}$, $\lambda^{z,1}$, $\lambda^{z+\frac{q}{2},0}$ are respectively constituents of 
\[\begin{cases} \Ind_A^G(\alpha),\quad \Ind_B^G(\beta'), \quad \Ind_C^G(\gamma), \quad \Ind_C^G(\gamma'), &\text{if $\frac{q}{2}$ is odd},\\
\Ind_A^G(\alpha), \quad\Ind_C^G(\gamma'), \quad\Ind_C^G(\gamma), \quad\Ind_A^G(\alpha'),& \text{if $\frac{q}{2}$ is even.}\end{cases}\] Likewise, for all $z \in \{kq : k \in \ZZ\}$, the characters $\nu^{z,0}$, $\nu^{z+\frac{q}{2},1}$, $\nu^{z,1}$, $\nu^{z+\frac{q}{2},0}$ are respectively constituents of 
\[\begin{cases}
 \Ind_A^G(\alpha),\quad\ \Ind_C^G(\gamma'''), \quad \Ind_C^G(\gamma''), \quad \Ind_B^G(\beta'), &\text{if $\frac{q}{2}$ is odd and $\frac{r}{2q}$ is even},\\
\Ind_A^G(\alpha'),\quad \Ind_C^G(\gamma'''), \quad \Ind_C^G(\gamma''), \quad \Ind_B^G(\beta), &\text{if $\frac{q}{2}$ is odd and $\frac{r}{2q}$ is odd},\\
 \Ind_A^G(\alpha), \quad\ \Ind_C^G(\gamma'''), \quad\Ind_C^G(\gamma''), \quad\Ind_A^G(\alpha'),& \text{if $\frac{q}{2}$ is even and $\frac{r}{2q}$ is even,}\\
  \Ind_A^G(\alpha'), \quad\Ind_C^G(\gamma'''), \quad\Ind_C^G(\gamma''), \quad \Ind_A^G(\alpha),& \text{if $\frac{q}{2}$ is even and $\frac{r}{2q}$ is odd.}
 \end{cases}\] Exactly as in the previous lemmas, this suffices  to show that  $\{\alpha,\alpha',\beta,\beta',\gamma,\gamma',\gamma'',\gamma'''\}$ is a GIM for $G$ by dimensional considerations.
\end{proof}

Combining these lemmas, we have the following theorem, which is equivalent to Theorem \ref{n=2-thm} in the introduction by Theorem \ref{thm2}.

\begin{thm}\label{n=2-cl}
$G(r,p,q,2)$ has a GIM if and only if one of the following  mutually exclusive conditions occurs:
\begin{enumerate}
\item[(i)] $p \equiv q \modu 2)$.
\item[(ii)] $p \not \equiv q \modu 2)$ and $\frac{r}{pq}$ is an odd integer.
\item[(iii)] $(r,p,q) = (4,1,2)$.
\end{enumerate}
\end{thm}

\begin{proof}
The preceding three lemmas show that $G(r,2,q,2)$ has a GIM if $q$ is even.  This implies that $G(r,p,q,2)$ has a GIM whenever  $p,q$ are both even because in this case $G(r,p,q,2)\cong G(r,2,\frac{pq}{2},2)$ by Theorem \ref{thm2}.  The theorem now follows from this observation and Lemma \ref{n2-lem}. 
\end{proof}

\def\GL{\mathrm{GL}}

\section{Generalized involution models for quotients groups}

Let $G$ be a finite group with a normal subgroup $N$, and suppose $\nu\in \Aut(G)$ is an automorphism such that $\nu^2=1$ and $\nu(N)=N$. Then $\nu$ defines  also an automorphism of the quotient group $G/N$, that we still call $\nu$. If $G$ has a GIM with respect to $\nu$, then it is not always true that $G/N$ also has a GIM with respect to $\nu$. Here, however, we give one sufficient condition for this to occur.

Following \cite{APR}, we use the term
\emph{Gelfand model} to refer to a representation of a group equivalent to the multiplicity-free sum of all of the group's irreducible representations.
Recall that the irreducible representations of the quotient $G/N$ are given exactly by the irreducible representations of $G$ whose kernel contains $N$. Therefore, if $\rho : G \to \GL(V)$ is a \emph{Gelfand model} for $G$ and $V^{N}=\{v\in V : \rho(n)(v)=v \text{ for all } n\in N\}$, then the obvious representation $\rho : G/N \to \GL(V^N)$ is a Gelfand model for $G/N$.

We exploit this observation in the following proposition. To state this result, let
\[ \mathcal V_{G,\nu}\eqdef \mathbb C\spanning\{C_\omega:\,\omega 	\in \cI_{G,\nu}\}\] be a complex vector space with a basis indexed by the generalized involutions $\cI_{G,\nu}$, and 
suppose $G$ has a GIM with respect to $\nu$. By \cite[Lemma 2.1]{Ma}, this is equivalent to the existence of a function $\phi : G \times \cI_{G,\nu} \to \CC^*$
such that the formula 
\[\rho(g)(C_\omega)\eqdef \phi(g,\omega)C_{g\cdot \omega \cdot \nu(g)^{-1}}\quad\text{for $g \in G$ and $\omega \in \cI_{G,\nu}$}\]
extends to a representation $\rho : G \to \mathrm{GL}(\cV_{G,\nu})$ which is a Gelfand model.
With respect to this notation, we have the following statement.

\begin{prop}\label{quotient-prop} Suppose $G$ has a GIM with respect to $\nu$, so that there exists a function $\phi : G \times \cI_{G,\nu} \to \CC^*$ defining a Gelfand model as above.  Assume that both of the following conditions hold:
\begin{enumerate}
\item[(i)] If $\omega \in \cI_{G,\nu}$ and $g \in C_{G,\nu}(\omega)\cap N$, then $\phi(g,\omega) = 1$.
\item[(ii)] If $\omega N\in \cI_{G/N,\nu}$ for some $\omega\in G$, then $\omega\in \cI_{G,\nu}$.
\end{enumerate}
Then the quotient group $G/N$ has a GIM respect to $\nu$.
\end{prop}

\begin{proof}
To prove the proposition we show that  $\mathcal V_{G,\nu}^N$ has a $\mathbb C$-basis $\{ C_{\bar \omega} : \bar \omega \in   \cI_{G/N,\nu}\}$ indexed by the generalized  involutions in $G/N$, on which the action of $G/N$ on $\mathcal V_{G,\nu}^N$ has the formula
\begin{equation}\label{quotient-claim}
\rho(\bar g)(C_{\bar \omega})=\psi(\bar g, \bar \omega)C_{\bar g \cdot \bar \omega \cdot \nu( \bar g )^{-1}},\quad\text{for a function $\psi : G/N \times \cI_{G/N,\nu} \to \CC^*$.}
\end{equation}
This suffices to show  that $G/N$ has a generalized involution model by  \cite[Lemma 2.1]{Ma}.
To this end, we define an operator $\Phi  \in \End(\cV_{G,\nu})$ by
\[
\Phi = \frac{1}{|N|} \sum_{ n \in N} \rho(n) \in \End(\cV_{G,\nu}),\]
and for each  $\omega \in \cI_{G,\nu}$ we let
\[  C^N_\omega \eqdef \Phi C_\omega \in \cV_{G,\nu}.\]
Since $\rho(n) \Phi = \Phi \rho(n)= \Phi$ for all $n \in N$, 
the elements $ C^N_\omega$   clearly belong to $\cV_{G,\nu}^N$ for all $\omega \in \cI_{G,\nu}$; furthermore, they span the subspace $\cV_{G,\nu}^N$ since if 
$v  \in \cV_{G,\nu}^N$ then $v = \Phi v$,  so as we can write $v$ as a linear combination of the basis elements $C_\omega$, we can do the same with the $ C^N_\omega$'s.

For all $n \in N$, we have
\begin{equation}\label{scalar} \phi(n,\omega)  C^N_{n\cdot \omega \cdot \nu(n)^{-1}} =\Phi \( \phi(n,\omega) C_{n\cdot \omega \cdot \nu (n)^{-1}}\)= \Phi \rho(n) C_{\omega} =  C^N_\omega.\end{equation}
Thus if we let $\cR$ be a set of representatives of the distinct $N$-orbits in $\cI_{G,\nu}$, then the set $\{ C^N_\omega : \omega \in \cR\}$ also spans $\cV_{G,\nu}^N$.
Each $ C^N_\omega$ is a  linear combination of the vectors $C_{\omega'}$ for $\omega'$ in the $N$-orbit of $\omega$.  
Therefore, to prove  that 
$\{ C^N_\omega : \omega \in \cR\}$ is also linearly independent and hence a basis for $\cV_{G,\nu}^N$ it is enough to show that $ C^N_\omega \neq 0$ for all $\omega \in \cI_{G,\nu}$. 
This follows since  the coefficient of $C_\omega$ in $ C^N_\omega$ is 
\[\frac{1}{|N|} \sum_{g \in C_{G,\nu}(\omega) \cap N} \phi(g,\omega) = \frac{|C_{G,\nu}(\omega)|}{|N|} \neq 0\] by condition (i).

Let $s : \cI_{G,\nu} \to \cR$ be the map which assigns $\omega' \in \cI_{G,\nu}$ to the unique $\omega \in \cR$ in the same $N$-orbit.  Then $\omega'$ and $\omega = s(\omega')$ also belong to the same left coset of $N$, since  
by definition 
\[\omega' =n \cdot \omega\cdot \nu(n)^{-1} = \omega \cdot \underbrace{(\omega^{-1} n \omega ) \cdot \nu(n)^{-1}}_{\in N}\qquad\text{for some $n \in N$}.\]
By condition (ii), it follows  
that  the map $  \omega N  \mapsto s(\omega)$ is  a well-defined bijection $\cI_{G/N,\nu} \to \cR$, and so we may define $C_{ \omega N} =  C^N_{s(\omega)}$ for $ \omega N \in \cI_{G/N,\nu}$. As noted above, the elements $C_{\omega N}$ for $ \omega N \in \cI_{G/N,\nu}$  form a basis for $\cV_{G,\nu}^N$.
Noting that $\Phi$ commutes with $\rho(g)$ for every $g \in G$ and that $ C^N_{\omega}$ is equal  by (\ref{scalar}) to a nonzero constant times $ C^N_{s(\omega)}$ for each $\omega \in \cI_{G,\nu}$, one checks that the action of $G/N$ on 
the basis elements $C_{ \omega N}$ has the prescribed form \eqref{quotient-claim}, which completes our proof.
\end{proof}

The motivating application of this proposition is the following corollary, which establishes part (1) of Theorem \ref{thm3} in the introduction.

\begin{cor}\label{part1-cor}
  If $G(r,p,n)$ has a GIM and $q$ is odd, then $G(r,p,q,n)$ has a GIM. In particular, if $\gcd(p,n) = 1$ and $q$ or $n$ is odd, then $G(r,p,q,n)$ has a GIM.
\end{cor}
\begin{proof}
   Let $G=G(r,p,n)$, take $N$ to be the cyclic subgroup  of order $q$ generated by $c^{r/q} \in G$, and let $\tau \in \Aut(G)$ be the usual automorphism $(\pi,x) \mapsto (\pi,-x)$.  Assume $G$ has a GIM; by \cite[Lemmas 5.1 and 5.2]{Ma}, this GIM may be defined with respect to $\tau$.
We have
\[
   c^{ir/q}\cdot \omega\cdot \tau(c^{-ir/q}) =c^{2ir/q}\omega \qquad\text{for }\omega \in \cI_{G,\tau},
\]
and so $c^{ir/q}\in C_{G,\tau}(\omega)$ if and only if $2ir/q\equiv 0 \modu r)$. Since $q$ is odd this is equivalent to $i$ being divisible by $q$, so  $C_{G,\tau}(\omega) = \{1\}$ and condition (i) in Proposition \ref{quotient-prop} holds automatically. Condition (ii) likewise follows, from  \cite[Lemma 4.2]{Ca2}, so $G/N = G(r,p,q,n)$ has a GIM.

To prove the second statement in the corollary suppose $\gcd(p,n) = 1$. Then $G(r,p,n)$ has a GIM by Theorem \ref{thm1} so   $G(r,p,q,n)$ has a GIM if $q$ is odd. If $q$ is even but $n$ is odd, then   $q'\eqdef q/2^k$ is an odd integer for some positive integer $k$. By Theorem \ref{thm2}  we then have $G(r,p,q,n) \cong G(r,p',q',n)$ where $p'\eqdef 2^k p$; since $\gcd(p',n) = 1$ we likewise conclude that $G(r,p,q,n)$ has a GIM.
\end{proof}

\section{Conjugacy classes and characteristic subgroups}

\def\supp{\mathrm{supp}}

Here we prove a few miscellaneous results which will be of use in the next sections. As usual we let $r,p,q,n$ be positive integers with $p$ and $q$ dividing $r$ and $pq$ dividing $rn$.

First, we  define a ``colored cycle decomposition'' of an element $g\in G(r,n)$ in the following way.
Recall that the \emph{support} of a permutation $\pi \in S_n$, which we denote $\supp(\pi)$, is the set of $i \in [n]$ with $\pi(i) \neq i$. A \emph{cycle} in $S_n$ is a nontrivial permutation which acts transitively on its support; i.e., an element of the form $(i_1,i_2,\dots,i_l) \in S_n$ with $l \geq 2$. (Note we do not consider the identity to be a cycle.)
%
%
Call $\gamma \in G(r,n)$ a \emph{colored cycle} if either 
\begin{enumerate}
\item[(i)] $|\gamma| \in S_n$ is a  cycle and $z_i(\gamma) = 0$ for all $i \notin \supp(|\gamma|)$; 
\item[(ii)] $|\gamma| = 1$ and $z_i(\gamma)  \neq 0$ for exactly one index $i \in [n]$.
\end{enumerate}
We define the \emph{support} of a colored cycle to be the set of $i \in [n]$ such that either $i \in \supp(|\gamma|)$ or $z_i(\gamma) \neq 0$; we denote this set also as $\supp(\gamma)$.
This definition ensures that the following 
desirable property holds:

\begin{lem} Each
 $g  \in G(r,n)$ has a unique factorization as a product of disjoint colored cycles; i.e., there is a unique set of colored cycles $\{\gamma_1,\dots,\gamma_k\} \subset G(r,n)$ such that $\supp(\gamma_i)\cap\supp(\gamma_j ) = \varnothing$ for all $i\neq j$ and $g = \gamma_1 \gamma_2 \cdots \gamma_k$.
 \end{lem}
 
 \begin{proof}
 The proof is a straightforward exercise left to the reader.
 \end{proof}

Define the \emph{length} of a colored cycle $\gamma \in G(r,n)$ to be the size of its support. Likewise define the \emph{color} of the cycle $\gamma \in G(r,n)$ to be the integer $\Delta(\gamma) \in \ZZ_r$. Note that while all cycles in $S_n$ have length at least two, we allow colored cycles of any positive integer length.

We now define a ``splitting index'' controlling how the $G(r,n)$-conjugacy class of an element $g$ decomposes into $G(r,p,n)$-conjugacy classes.
First, given $\gamma \in G(r,n)$ a colored cycle of length $l$ and color $a$, we define an integer 
\[ s(\gamma) \in \ZZ_{\GCD(a,l,r)}\] according to the following cases:
\begin{itemize}
\item If  $|\gamma| = (i_1,i_2,\dots,i_l)$ is a  cycle in $S_n$, then let
\[ s(\gamma) \eqdef \sum_{j=1}^l j z_{i_j}(\gamma) \in \ZZ_{\GCD(a,l,r)}.\]
\item If $|\gamma|=1$ then we define $s(\gamma) \eqdef 0 \in \ZZ_1 = \ZZ_{\GCD(a,l,r)}$.
\end{itemize}
This definition of $s(\gamma)$ does not depend on the ordering of the cycle $(i_1,i_2,\dots,i_l)$, 
because if $x_1,x_2,\dots,x_l\in \ZZ$ are any integers  and $a=x_1+x_2+\cdots+x_l$, then we have
\begin{equation}\label{somma}
   \sum_{j=1}^ljx_j \equiv \sum_{j=1}^ljx_{j+1} \modu \GCD(a,l)),\qquad\text{where we let }x_{l+1}\eqdef x_1.
\end{equation}
 In particular, one can check that the difference of the two sides of this congruence is $\pm(a-lx_1 )$. 

Expanding on this notation, we have  the following definition.
\begin{defn}
Fix a positive divisor $p$ of $r$ and an  element $g \in G(r,p,n)$. Suppose $g$ factors as the product of the disjoint colored cycles $\gamma_1,\gamma_2,\dots,\gamma_k \in G(r,n)$, where $\gamma_i$ has length $l_i$ and color $a_i$.
We then let 
\[ d_p(g)\eqdef\GCD(l_1,l_2,\ldots,l_k,a_1,a_2,\ldots,a_k,p)\] and we define the \emph{$p^{\mathrm{th}}$ splitting index} of $g$ to be 
\[
   s_p(g)\eqdef \sum _{i=1}^k s(\gamma_i) \in \mathbb Z_{d_p(g)}.
\]
Note that the right hand side  makes sense since $d_p(g)$ divides $\GCD(l_i,a_i,r)$ for each $i$, and that $s_p(g) = 0$ if $g$ has any colored cycles of length one.
%
 As the sum has no dependence on the labeling of the colored cycles $\gamma_i$, the splitting index $s_p(g)$ is automatically well-defined.
\end{defn}

Before describing the significant properties of this definition,
first let us consider an example.

\begin{exa} Let $g= (\sigma,x) \in G(4,4,8)$ where 
\[\sigma = (1,2,4,7)(3,5,8,6) \in S_8 \qquad\text{and}\qquad x = (0,1,2,2,0,2,1,0) \in (\ZZ_4)^8.\] We then take $\gamma_1=\(e_2 + 2e_4 + e_7, (1,2,4,7)\)$ and $\gamma_2=\(2e_3  +2 e_6, (3,5,8,6)\)$, giving  $a_1=a_2=0$ and $d_p(g)=4$ and 
\[\ba s_p(g)&=s(\gamma_1) + s(\gamma_2) \\&=\(1\cdot 0+2\cdot 1+3\cdot 2+4\cdot 1\)+\(1\cdot 2+2\cdot 0+3\cdot 0+4\cdot 2\)
\\
&=2\in \ZZ_{4}.
\ea
\]
\end{exa}

The following result  indicates the utility of this notion of splitting index.
In the special case when $\GCD(p,n) \leq 2$, this proposition coincides with \cite[Lemma 5.1]{CaF}.

\begin{prop}\label{grpnconjugacy}Let $g,g'\in G(r,p,n)$ and $h\in G(r,n)$. The following properties then hold:
\begin{enumerate}
\item[(1)] $d_p(h^{-1} gh) = d_p(g)$.
\item[(2)] $s_p(h^{-1}gh)=s_p(g)+\Delta(h)\in \mathbb Z_{d_p(g)}$. 
\item[(3)] The elements $g$ and $g'$ are conjugate in $G(r,p,n)$ if and only if they are conjugate in $G(r,n)$ and $s_p(g)=s_p(g').$
\end{enumerate}
\end{prop}
\begin{proof}
Throughout, we write  $\gamma_1,\dots,\gamma_s$ for the colored cycles of $g$, where $\gamma_i$ has length $l_i$ and color $a_i$.
Part (1) is an immediate consequence of the observation that conjugating a colored cycle in $G(r,n)$ preserves its length and color.

To prove part (2), note that if $\pi \in S_n$ then the colored cycles of $\pi^{-1} g \pi$ are precisely $\pi^{-1} \gamma_i \pi$ for $i \in [s]$, and we have \[\supp(\pi^{-1}\gamma_i \pi) = \pi^{-1} \( \supp(\gamma_i)\) \qquad\text{and}\qquad z_{\pi^{-1}(j)} (\pi^{-1} \gamma_i \pi) = z_{j}(\gamma_i)\]
for each $j \in [n]$.
It follows that $s(\gamma_i) = s(\pi^{-1} \gamma_i \pi)$ for each $i$, whence $s_p(\pi^{-1} g \pi) = s_p(g)$. To prove $s_p(h^{-1}gh)=s_p(g)+\Delta(h)$ we may therefore assume $|h| = 1$.  The colored cycles of $h^{-1} gh$ are then $h^{-1} \gamma_i h$ for $i \in [s]$, and the support of $h^{-1} \gamma_i h$ 
coincides with that of $\gamma_i$.
It is therefore enough to prove that if $\gamma$ is an arbitrary colored cycle in $G(r,n)$ of length $l$ and color $a$ and $|h|=1$, then
\[s(h^{-1}\gamma h) = s(\gamma) + \sum_{j \in \supp(\gamma)} z_j(h) \in \ZZ_{\GCD(a,l,r)}.\]
If $l=1$ then this equality holds vacuously, while if $l>1$ so that $|\gamma| = (i_1,i_2,\dots,i_l) \in S_n$, then one has $z_{i_j}(h^{-1} \gamma h) = z_{i_j}(\gamma) + z_{i_j}(h) - z_{i_{j+1}}(h)$ for each $j \in [l]$ where we define $i_{l+1} \eqdef i_1$. Thus 
\[ s(h^{-1} \gamma h) = \sum_{j=1}^l j z_{i_j} (\gamma) + \(\sum_{j=1}^l j z_{i_j}(h) - \sum_{i=1}^l j z_{i_{j+1}}(h)\) \in \ZZ_{\GCD(a,l,r)}.\] On the right, the first term is just $s(\gamma)$, while one computes that the parenthesized sum is $\sum_{j \in \supp(\gamma)} z_j(h) - l z_{i_1}(h) = \sum_{j \in \supp(\gamma)} z_j(h)  \in \ZZ_{\GCD(a,l,r)}$. Thus our claim holds, and the second  part follows.

Finally, to prove part (3) we note from part (2) that  if $g$ and $g'$ are conjugate in $G(r,p,n)$ then they have the same $p^{\mathrm{th}}$ splitting index.
Conversely, assume that $g' = h^{-1} g h $ and $s_p(g)=s_p(g')$. By part (2) we then have $\Delta(h)\equiv 0 \modu d_p(g))$, so
to show that $g$ and $g'$ are conjugate in $G(r,p,n)$ it suffices to produce an element $\xi\in G(r,n)$ which centralizes $g$ and has $\Delta(\xi)\equiv d_p(g) \modu p)$. In particular, for such a $\xi$ one would  have $\xi^j h \in G(r,p,n)$ for some $j$ and also $g ' =(\xi^j h)^{-1} g (\xi^j h)$.
To construct $\xi$,  let $n_1,\dots,n_s,m_1,\ldots,m_s,k$ be integers such that
\[
   d_p(g)=n_1l_1+\cdots+n_sl_s+m_1a_1+\cdots+m_sa_s+kp.
\]
For each $i \in [s]$ let $t_i\in G(r,n)$ be such that $|t_i|=1$ and $z_j(t_i)=1$ if $j\in \supp(\gamma_i)$ and $z_j(t_i)=0$ otherwise.
Now  let $\xi=t_1^{n_1}\cdots t_s^{n_s}\gamma_1^{m_1}\cdots \gamma_s^{m_s}$. It is clear that $\xi$ centralizes $g$ and that $\Delta(\xi)=n_1l_1+\cdots+n_sl_s+m_1a_1+\cdots+m_sa_s\equiv d_p(g) \modu p)$, which completes our proof.
\end{proof}

Our proof of the following lemma provides one application of the preceding proposition.
Here, given an element $g \in G(r,1,q,n)$, we write $\Ad(g)$ for the automorphism of $G(r,p,q,n)$ defined by $h\mapsto ghg^{-1}$.

\begin{lem}\label{new4.3}
   Let $g\in G(r,1,q,n)$  such that $\Ad(g)(\pi)$ and $\pi$ are conjugate in $\qc$ for all $\pi\in S_n$. Then either 
   \begin{enumerate}
   \item[(i)] $\Ad(g)=\Ad(h)$ for some $h\in G(r,p,q,n)$;
   \item[(ii)] $\Ad(g)=\Ad(h)$ for some $h\in G(r,p/2,q,n)$, and 
   \[r\equiv p\equiv q\equiv n\equiv 2^i \modu 2^{i+1})\] for an integer $i>0$.
   \end{enumerate}
\end{lem}

\begin{proof}
Without loss of generality we may  assume  $g=t^a$ for an integer $a$, since elements of this form represent the distinct left cosets of $G(r,p,q,n)$ in $G(r,1,q,n)$. 
We wish to show  $\Ad(t^a) = \Ad(h)$ for an element $h$ of either $G(r,p,q,n)$ or $G(r,p/2,q,n)$. 

To this end, 
let $\pi=(1,2,\ldots,n)^{-1}$ be the inverse of the standard $n$-cycle in $S_n$.
By hypothesis, $\pi$ is then conjugate in $G(r,p,n)$ to $\Ad(t^a)(\pi) c^{\frac{kr}{q}}$ for some $k \in \{1,\dots,q\}$.  
Since  $\Delta$ is a homomorphism and $\Delta(\pi) = 0$, it follows that $\Delta(c^{\frac{kr}{q}})=0$ which implies $q$ is  a divisor of $nk$.

It is straightforward to compute, using the first part of the previous proposition, that \[d_p(\Ad (t^a)(\pi) c^{\frac{kr}{q}}) = d_p(\pi c^{\frac{kr}{q}}) = d_p(\pi)=\GCD(p,n)\] and that
$ s_p(\Ad (t^a)(\pi) c^{\frac{kr}{q}})=a+\tfrac{(n+1)nkr}{2q}
$
and 
$s_p(\pi)=0$.
Part (3) of Proposition \ref{grpnconjugacy} therefore reduces to the congruence
\begin{equation}\label{a-cong}
 a+\tfrac{(n+1)nkr}{2q} \equiv 0 \modu \GCD(p,n)).
\end{equation}
We wish to deduce from this  that either (i) $a $ is a multiple of $\GCD(p,n)$ or that (ii) $r\equiv p\equiv q\equiv n\equiv 2^i \modu 2^{i+1})$ for an integer $i>0$.
Since $q$ divides both $nk$ and $r$ and since $p$ divides $r$,  
the integer
$\tfrac{(n+1)nkr}{2q}$ is a multiple of $\GCD(p,n)$ if any of the integers $nk/q$  or $n+1$ or $r/p$ or $r/q$ are even. From the congruence \eqref{a-cong}, it follows that (i) holds unless $k$ and $nk/q$  and $n+1$ and $r/p$ and $r/q$ are all odd, in which case (ii) holds.

To complete the
 proof, we note that if $a $ is a multiple of $\GCD(p,n)$, then
there exists an integer $j$ such that $a + jn$ is a multiple of $p$, in which case $\Ad(t^a) = \Ad(h)$ for the element $h = t^a c^j \in G(r,p,q,n)$.
On the other hand, if (ii) holds 
then  
$\frac{(n+1)nkr}{2q}$ is a multiple of $p/2$, which is in turn a multiple of $\GCD(p/2,n) =\GCD(p,n)/2$. In this case it follows from the congruence \eqref{a-cong}   that $a$ is  a multiple of $\GCD(p/2,n)$, so there exists an integer $j$ such that $a + jn$ is a multiple of $p/2$, and we have $\Ad(t^a) = \Ad(h)$ for the element $h = t^a c^j \in G(r,p/2,q,n)$.
\end{proof}

Switching topics briefly, we now prove a result which, while not strictly needed, extends the scope of some of our proofs in the next section.
Recall that a subgroup $H$ of a group $G$ is \emph{characteristic} if $H$ is invariant under all automorphisms of $G$. We note that if $H$ is an abelian normal subgroup of $G$, then the image of $H$ under any automorphism of $G$ is also abelian and normal. Therefore, if $G$ has a \emph{unique}  abelian normal subgroup $H$ of a given size, then $H$ is necessarily characteristic. We will frequently appeal to this fact in the  proof of the following.

\begin{prop}\label{characteristic-prop} 
The diagonal subgroup $N = N(r,p,q,n)$ is  characteristic in $G=\qc$ if and only if $(r,p,q,n)$ is not one of the following twelve exceptions:
\begin{itemize}
\item[]
$(2,1,1,2)$ or
$(2,2,1,2)$ or
$(2,1,2,2)$ or

\item[]
$(4,1,2,2)$ or 
$(4,2,1,2)$ or
$(4,2,2,2)$ or
$(4,2,4,2)$ or
$(4,4,2,2)$ or

\item[]
$(3,3,1,3)$ or
$(3,3,3,3)$ or

\item[]
$(2,2,1,4)$ or
$(2,2,2,4)$.

\end{itemize}

\end{prop}
\begin{proof}
If $r=1$ then $N$ is trivial and if $n=1$ then $N=G$. In either case $N$ is automatically characteristic, so assume $r>1$ and $n>1$. Since  $N$ is an abelian  normal subgroup of $G$, to show that $N$ is characteristic, it suffices to show that every other abelian normal subgroup of $G$ contains fewer elements than $N$.
 
   If $n\geq 5$ then the only  abelian normal subgroup of $S_n$ is  trivial, so  every abelian normal subgroup of $G$ must  be contained in $N$, since this is the kernel of the natural surjection $|\cdot| : G \to S_n$. Therefore $N$ is characteristic. 

 Next, suppose $n=4$. The symmetric group $S_4$ has a unique nontrivial normal abelian subgroup $V$ given by the set 
 \[ V=\{1,\sigma_2,\sigma_3,\sigma_4\},\] where $\sigma_i$ is the unique fixed-point free involution in $S_4$ with  $\sigma_i(1)=i$. Suppose $H$ is an abelian normal subgroup of $G$ not contained in $N$. The image of $H$ under $|\cdot|: G \to S_4$ must then be equal to  $V$, and so if $g,h \in H$ such that $|g| = |h|$, then since $H$ is abelian $gh^{-1}$ must belong to the centralizer of $V$ in $N$; i.e., $gh^{-1}$ belongs to the subgroup
 \[ C_N(V)  \overset{\mathrm{def}} = \{a \in N : ava^{-1} = v\text{ for all }v \in V\}.\]
 In particular, this means that 
$ |H| \leq |V| |C_N(V)| = 4 |C_N(V)|.$
 It is straightforward to compute that 
\begin{itemize}
\item $C_N(V)=C(r,p,q,4)$ if $q$ is odd;
\item $C_N(V)=C(r,p,q,4)\times V'$ if $q$ is even, where $V'$ is the four-element subgroup  generated by $(1,(\tfrac{r}{2},\tfrac{r}{2},0,0))$ and $(1,(\tfrac{r}{2},0,\tfrac{r}{2},0))$ in $N$.
\end{itemize}
Thus, recalling that $|C(r,p,q,n)| = \frac{r}{pq} \gcd(p,n)$, we have 
\[ |H| \leq 4\cdot \tfrac{r}{pq} \cdot \gcd(p,4) \cdot \gcd(q,2)^2.\]
The order of $N$ is $\frac{r^4}{pq}$, and this exceeds the right hand side of the preceding inequality if
$(r,p,q)$ is not $(4,4,4)$, $(4,4,2)$, $(2,2,2)$, $(2,1,2)$, or $(2,2,1)$. If we are not in any of these five cases, consequently, it follows that $N$ is the unique abelian normal subgroup of $G$ of order $\frac{r^4}{pq}$, and therefore characteristic. 
In the five excluded cases, we have checked using the computer algebra system {\sc{Magma}} that $N$ is characteristic if and only if $(r,p,q)$ is $(4,4,4)$, $(4,4,2)$, or $(2,1,2)$.

Next, let $n=3$. The unique  nontrivial abelian subgroup of  $S_3$ is the cyclic group $A$ of order 3 generated by either three cycle. It follows as in the previous case that if $H$ is an abelian normal subgroup of $G$ not contained in $N$, then 
$ |H| \leq |A| |C_N(A)| = 3 |C_N(A)|.$
Here, one computes that 
\begin{itemize}
\item $C_N(A)=C(r,p,q,3)$ if $q$ is not divisible by three;
\item $C_N(A)=C(r,p,q,3)\times A'$ if $q$ is divisible by three, where $A'$ is the three-element subgroup  generated by $(1,(0,\tfrac{r}{3},\tfrac{2r}{3})) \in N$.
\end{itemize}
Thus, an abelian normal subgroup $H \not\subset N$ has \[ |H| \leq 3\cdot \tfrac{r}{pq} \cdot \gcd(p,3) \cdot \gcd(q,3),\]
and it follows that $N$ is the unique (hence characteristic) abelian normal subgroup of order $\frac{r^3}{pq}$ if
 $(r,p,q)$ is not $(3,3,3)$, $(3,1,3)$, or $(3,3,1)$. In the excluded cases, moreover, $N$ is characteristic in $G$ if and only if $(r,p,q) = (3,1,3)$.

Similarly, if $n=2$ and $H$ is an abelian normal subgroup of $G$ not contained in $N$,
then the $\{ |h| : h \in H\} = S_2$ and 
$|H| \leq |S_2| |C_N(S_2)| = 2 |C_N(S_2)|$. Computing the centralizer of $S_2$ in $N$ is somewhat more complicated than in the previous cases, but nevertheless one checks that if $t = (1,e_1) \in N$ and $c = (1,e_1+e_2) \in N$ then
\begin{itemize}
\item $C_N(S_2)=C(r,p,q,2)$ if $q$ is odd;
\item $C_N(S_2)=C(r,p,q,2) \cup t^{r/2} \cdot C(r,p,q,2)$ if $q$ is even and $r/p$ is even;
\item $C_N(S_2)=  C(r,p,q,2) \cup c^{p/4}t^{r/2} \cdot C(r,p,q,2)$ if $q$ is even, $r/p$ is odd, and $p$ is divisible by four; 
\item $C_N(S_2)=C(r,p,q,2)$ if $q$ is even, $r/p$ is odd, and $p$ is not divisible by four;
\end{itemize}
In each case, the order of $C_N(S_2)$ has order at most $\tfrac{r}{pq} \cdot \gcd(p,2) \cdot \gcd(q,2)$, and so the order of any abelian normal subgroup $H \not\subset N$ satisfies the inequality \[ |H| \leq 2\cdot \tfrac{r}{pq} \cdot \gcd(p,2) \cdot \gcd(q,2).\]
It follows that $N$ is the unique abelian normal subgroup of order $\frac{r^2}{pq}$ if $r\notin \{2,4,6,8\}$. If $r $ is one of these excluded values, one can checks (e.g., using {\sc{Magma}}) that $N$ is characteristic in $G$ if and only if $r >4$ or  $(r,p,q)$ is $(2,2,2)$, $(4,1,1)$, $(4,1,4)$, or $(4,4,1)$.
\end{proof}

\section{Automorphisms of $G(r,p,q,n)$}

This section contains a number of technical lemmas which together describe the form of an arbitrary automorphism of $G(r,p,q,n)$. 
Throughout we write $s_i  \eqdef (i,i+1) \in S_n$ for the simple transpositions in $S_n$
and
 let $s,t,c \in G(r,1,q,n)$ denote the elements 
\[s\eqdef (1,e_1-e_2), \qquad t\eqdef (1,e_1),\qquad \text{and}\qquad 
c\eqdef (1,e_1+e_2+\cdots+e_n).\]
The elements $s_1,\dots,s_{n-1}, s, t^p$ then generate $G(r,p,q,n)$; note that the central element $c$ coincides with the identity if $r=q$. As in Section \ref{prelim-sect} we let
$N(r,p,q,n)$ denote the abelian normal subgroup of elements of the form $(1,x)\in G(r,p,q,n)$. We also define $C(r,p,q,n)$ as the subgroup of $\qc$ consisting of elements equal to $c^i$ for some $i \in \ZZ$.

Our first proposition establishes the existence of a generic type of outer automorphism of $G(r,p,q,n)$.
Recall here that $C_q$ denotes the cyclic subgroup of scalar matrices of order $q$ in $G(r,p,n)$, so that $G(r,p,q,n) = G(r,p,n)/C_q$. 

\begin{prop}\label{alpha-thm}
Assume $n\geq 3$ and let $d_0= \GCD(p,q,n)$. Suppose $j,k\in \mathbb Z_r$ and $z\in C(r,p,q,n)$ such that $z^2=1$. 

\begin{enumerate}
\item[(1)] The map $\alpha_{j,k,z} : G(r,p,n) \to \qc$ given by
\[
\alpha_{j,k,z}(\pi,x)=z^{\ell(\pi)}c^{\frac{\Delta(x)}{d_0} k}(\pi,jx),\qquad\text{for }(\pi,x) \in G(r,p,n),
\]
is a well-defined homomorphism whose kernel contains $C_q$. 

\item[(2)] The induced homomorphism $\alpha_{j,k,z}:\qc\rightarrow\qc$ is an automorphism if and only if
\begin{equation}\label{jrd0}
\GCD(j,r)=\GCD\Big(\tfrac{n}{d_0}k+j,\tfrac{rn}{pq},\tfrac{r}{q}\Big)=1.
\end{equation}
\end{enumerate}
\end{prop}

\begin{proof}
If $q=1$ this is exactly  \cite[Lemma 4.2]{Ma}.
The right-hand side of the definition of $\alpha_{j,k,z}(\pi,x)$ does not depend on the representative of $\Delta(x)$ chosen modulo $r$ because $c^{\frac{r}{d_0}}$  is a power of  $c^{\frac{r}{q}}$ and hence equal to the identity in $\qc$.
Similarly, 
 $\alpha_{j,k,z}(\pi,x)$ is a well-defined element of $G(r,p,q,n)$ because \[\Delta\(c^{\frac{\Delta(x)}{d_0} k}\)=\tfrac{n}{d_0}\Delta(x)k\] is a multiple of $\Delta(x)$ and hence of $p$.
We conclude that the map $\alpha_{j,k,z}$ is well-defined, and proving that it is a homomorphism is an easy exercise left to the reader.

In what follows we abbreviate by defining $\alpha=\alpha_{j,k,z}$.
We have $\ker \alpha \supset C_q$ since
\[
\alpha(c^{r/q})=c^{\frac{rn}{d_0q}k}\cdot c^{j\frac{r}{q}}=c^{(\frac{n}{d_0}k+j)\frac{r}{q}} = 1 \in \qc,\]
 so $\alpha$ descends to  a well-defined homomorphism $\alpha:\qc \to \qc$. 
Suppose $\alpha$ is an automorphism of $\qc$. Since $n>2$ the element $s$ has order $r$ in $\qc$, and so we have $\GCD(j,r)=1$ since $\alpha(s)=s^j$. On the other hand, if $i_0\eqdef \frac{p}{\GCD(p,n)}$ then $c^{i_0}$ has order $\frac{r}{pq}\GCD(p,n)$ in $\qc$ and generates the subgroup $C(r,p,q,n)$. The element
\[
   \alpha(c^{i_0})=c^{(\frac{n}{d_0}k+j)i_0}.
\]
 must also have order $\frac{r}{pq} \GCD(p,n)$, and this  occurs if and only if
\[
   \GCD\Big(\tfrac{n}{d_0}k+j,\tfrac{r}{pq}\GCD(n,p)\Big)=\GCD\Big(\tfrac{n}{d_0}k+j,\tfrac{rn}{pq},\tfrac{r}{q}\Big)=1.
\]
Hence if $\alpha$ is an automorphism then equation \eqref{jrd0} holds. Furthermore, if \eqref{jrd0} holds then $\alpha$ restricts to an automorphism of $C(r,p,q,n)$.

Conversely, suppose \eqref{jrd0} holds. To show that $\alpha$ is an automorphism it is enough to prove injectivity. 
If $(\pi,x)\in\ker(\alpha)$ then we clearly have $\pi=1$ and $jx$ must be of the form $ja(e_1+\dots + e_n)$ for some $a \in \ZZ_r$. Since $\GCD(j,r)=1$, we have $x= a(e_1+\dots +e_n)$, and so $(\pi,x)\in C(r,p,q,n)$. As we have already observed that 
\eqref{jrd0} implies that $\alpha$ restricts to an automorphism of $C(r,p,q,n)$,
  we have $(\pi,x)=1$.
\end{proof}

We will also require the following  construction of a certain exceptional automorphism of $G(r,p,q,4)$.

\begin{prop}\label{speaut4} 
Let $r,p,q$ be positive integers with $p$ and $q$ dividing $r$ and $pq$ dividing $4r$.
 Assume in addition that
\begin{enumerate}
\item[(i)] $q$ is even;
\item[(ii)]  if $r/p$ is odd, then  $r/q$ is even and  $\frac{4r}{pq}$ is odd.
\end{enumerate}
 Then there exists a unique automorphism $\phi$ of $G(r,p,q,4)$ with
\begin{equation}
\label{phi}\begin{aligned} \phi(s_i)&=\left\{\begin{array}{ll}t_{i+2}\cdot s_i&\text{if $r/p$ is even,}\\ t_{i+2} \cdot s_i\cdot  c^{\frac{r}{2q}}&\text{if $r/p$ is odd, }\end{array}\right.
&&\qquad\text{for $i \in \{1,2,3\}$,}
\\
\phi(x) &= x&&\qquad\text{for $x \in N$},
\end{aligned}
\end{equation}
where $t_j=(1,e_j)^{r/2} \in N(r,p,q,4)$ for  $j \in \{1,2,3,4\}$, and  where the indices are taken modulo 4 (so that $t_5 = t_1$).
\end{prop}

\begin{proof}
Let $G=G(r,p,q,4)$ and $N =N(r,p,q,4)$. To show that our formula for $\phi(s_i)$ 
is a well-defined element of $G$ it is enough to check that $\frac{r}{2}$ is multiple of $p$ if $r/p$ is even and that $\frac{r}{2} + \frac{4r}{2q}$ is a  multiple of $p$ if $r/p$ is odd. The first assertion is immediate, and the second follows by assumption (ii) since $\frac{r}{2} + \frac{4r}{2q}=\Big(\frac{r}{p}+\frac{4r}{pq}\Big)\frac{p}{2}$. 

We now show that $\phi$ extends to a homomorphism when  $r/p$ is even; the  case when $r/p$ is odd follows similarly. To this end, we claim that the formula
$s_i \mapsto t_{i+2}s_i$ defines a homomorphism $\phi: S_4 \to   G$.
It is straightforward to check that the images of the generators $s_1,s_2,s_3 \in S_4$ satisfy the relevant Coxeter relations, 
after  noting that $t_1t_2=t_3t_4$ since $q$ is even and observing that $t_j$ and $s_i$ commute unless $j=i+\epsilon$ for some $\epsilon \in \{0,1\}$,  in which case case $t_{i+\epsilon}s_i=s_it_{i+1-\epsilon}$.

Now, every $g\in G$ can be written uniquely as $g= \sigma\cdot x = (\sigma,x)$ with $x\in N$ and $\sigma\in S_4$,
and  one checks that the formula $ \sigma\cdot x \mapsto  \phi(\sigma)\cdot x$ defines an endomorphism of $G$, using the fact that
  $\phi(s_i) x \phi(s_i)^{-1} = s_i x s_i^{-1}$ for all $x \in N$, whence $\phi(\sigma) x \phi(\sigma)^{-1} = \sigma x \sigma^{-1}$ for all $\sigma \in S_4$ and $x \in N$.
This endomorphism is clearly injective 
and is therefore an automorphism.
\end{proof}

For the duration of this section we write $G = G(r,p,q,n)$ and $N=N(r,p,q,n)$ where $r,p,q,n$ are fixed positive integers 
with $p$ and $q$ dividing $r$ and $pq$ dividing $rn$.
In the following lemma,
recall that   $|\cdot| : G\to S_n$ denotes  the standard projection map  $(\pi,x) \mapsto \pi$.

\begin{lem}\label{Nchar} 
   Let $\nu\in \Aut(G)$   with $\nu(N)=N$ and define $\bar \nu : S_n \to S_n$ as the map 
   \[ \bar \nu (\pi) = |  \nu( \pi,0)|\quad\text{for } \pi \in S_n.\]
   The following properties then hold:
   \begin{enumerate}
   \item[(1)] The map $\bar \nu$ is an automorphism of $S_n$.
   \item[(2)] The automorphism $\bar \nu \in \Aut(S_n)$ is inner 
if $(r,p,q,n) \neq (1,1,1,6)$.
\end{enumerate}
\end{lem}
\begin{proof} 
The map $\bar \nu$ is automatically a homomorphism $S_n \to S_n$. It  is injective (and hence an automorphism) since $\nu(g) \in N$ if and only if $g \in N$.

Since   $S_n$ has no outer automorphisms if $n\neq 6$, to prove the lemma it suffices to  
assume that $n=6$ and that $\bar \nu$ is not inner, and argue that $r=p=q=1$.  In this case, the image of the 2-cycle $(1,2) \in S_6$ under $\bar \nu$ is necessarily a product of three disjoint 2-cycles $(i_1, i_2)(i_3,i_4)(i_5,i_6)$, as $S_6$ has only one nontrivial coset of outer automorphisms.
Since $\nu(N) = N$,
the centralizers $C_N((1,2))$ and $C_N((i_1, i_2)(i_3,i_4)(i_5,i_6))$ have the same order (the first subgroup is the image of the second under $\nu$).
On the other hand, one computes that the first group has $\frac{r^5}{pq}$ elements, while the second has either $2\frac{r^3}{pq} $ elements (if $p$ and $q$ are both even but $r/2$ is odd) or $\frac{r^3}{pq}\cdot \gcd(p,2)\cdot \gcd(q,2)$ elements (in all other cases).  As these numbers are equal, we must have $r=p=q=1$.
\end{proof}

We say that an automorphism $\nu \in \Aut(G)$ \emph{preserves the projection to $S_n$} if for all $(\pi,x)\in G$ there exists $y\in \mathbb Z_r^n$ such that $\nu(\pi,x)=(\pi,y)$. 
As we see in our next lemma, this property places strong conditions on the form of $\nu$.

\begin{lem}\label{prepro}
   Assume $n\geq  3$ and suppose $\nu\in \Aut(G)$ is an automorphism that preserves the projection to $S_n$. Then there exists $z\in C(r,p,q,n)$ with $z^2=1$ and  integers $a_1,a_2,\dots, a_{n-1}$ such that
\[
   \phi\circ \nu(s_i)=z\left( s_i, a_i e_i - a_i e_{i+1}\right),\qquad\text{for all }i \in [n-1],
\]
where either 
\begin{enumerate}
\item[(i)] $\phi$ is the identity automorphism;
\item[(ii)] $\phi$ is defined as in Equation \eqref{phi}, in the case that $n=4$ and  $(r,p,q)$ satisfy  the hypotheses of Proposition \ref{speaut4}.
\end{enumerate}
\end{lem}

\begin{proof} Because the elements $s_1,s_2,\dots,s_{n-1}$ are all conjugate in $G$, it is enough to show that the lemma holds for $i=1$.   We write $\nu(s_i)=( s_i,x_{i,1}e_1 +\dots +x_{i,n} e_n)$ and proceed as follows.

 For each $j \in \{3,4,\dots,n-1\}$, we have $\left(\nu(s_j)\nu(s_1)\right)^2 = 1$ which means that
   \[  (x_{j,1} - x_{j,2})(e_1 -e_2) + (x_{1,j+1} - x_{1,j})(e_j -e_{j+1}) \in \ZZ_r\spanning\left\{ \tfrac{r}{q} (e_1+\dots + e_n)\right\}.\]
If $n\geq 5$ then this containment implies  $x_{1,j} \equiv x_{1,j+1} \modu r)$ for each $2 <j<n$ 
since then
$\{1,2,j,j+1\} \neq \{1,2,\dots,n\}$. If $n=4$, alternatively, then we are only able to deduce that $a \eqdef x_{1,3}-x_{1,4}$ is a multiple of $r/q$ and that  $2a$ is multiple of $r$.
These observations show that for any $n\geq 3$ one of the following holds:
   \begin{enumerate}
   \item[(a)] $x_{1,3}=x_{1,4}=\dots = x_{1,n}$;
   \item[(b)] $x_{1,3}=x_{1,4} + \frac{r}{2}$ and $n=4$ and $q$ is even.
\end{enumerate}
   In either case,  since $s_1^2=1$, there exists an integer $k$ such that $x_{1,1}+x_{1,2}\equiv 2x_{1,3}\equiv k\frac{r}{q} \modu r)$.  We may assume that $x_{1,3}\in\{0,1,\ldots,\frac{r}{q}-1\}$, and so either
   \begin{itemize}
   \item $k=0$ and $x_{1,1}+x_{1,2}=0$ and $x_{1,3}=0$;
   \item $k=1$ and $x_{1,1}+x_{1,2}=\frac{r}{q}$ and $x_{1,3}=\frac{r}{2q}$ and $\frac{r}{q}$ is even.
   \end{itemize}
   If we are in case (a), then $k=0$ implies that $\nu(s_1) = (s_1,a_1e_1-a_1e_2)$ for $a_1 = x_{1,1}$, while $k=1$ implies that $\nu(s_1) = z(s_1,a_1e_1-a_1e_2)$ for $a_1 = x_{1,1}-\frac{r}{2q}$ and $z = c^{\frac{r}{2q}}$.    
   
   On the other hand, suppose we are in case (b). Then $k=0$ implies that $\frac{r}{p}$ is even (since in this case $\frac{r}{2}$  must be a multiple of $p$) 
   while $k=1$ implies that $\frac{2r}{q} + \frac{r}{2}$ is a multiple of $p$ (as otherwise $\nu(s_1)$ would not belong to $G$), in which case $\frac{4r}{pq} + \frac{r}{p}$ is even and $\frac{4r}{pq}$ and $\frac{r}{p}$ have the same parity.
   In either situation $(r,p,q)$ satisfy the hypotheses of Proposition \ref{speaut4} so we may define $\phi$ by \eqref{phi}.
   One then checks that if $k=0$ then 
    $\phi\circ \nu(s_1) = (s_1,a_1e_1-a_1 e_2)$ for $a_1 = x_{1,1} + \frac{r}{2}$ while if $k=1$ then 
   $\phi\circ \nu(s_1) = z(s_1,a_1e_1-a_1 e_2)$
   for $a_1 = x_{1,1} + \frac{r}{2} - \frac{r}{2q}$ and $z = c^{\frac{r}{2q}}$.
   
In all 
 cases $\nu(s_1)$ has the desired form, which suffices to prove the lemma.
   \end{proof}

Our final result in this section gives an explicit  form  for all automorphisms of $G=G(r,p,q,n)$ which preserve the normal subgroup $N=N(r,p,q,n)$. 

\begin{thm}\label{vN=N}
Let $n \geq 3$ and assume $(r,p,q,n) \neq (1,1,1,6)$. 
  If 
$\nu\in \Aut(G)$  is an automorphism such that $\nu(N) = N$, then
\[ \nu=\Ad(g)\circ \phi \circ \alpha_{j,k,z}\]
for some $g \in G(r,1,q,n)$ and  some $j,k,z$  as in Proposition \ref{alpha-thm}, where either
\begin{enumerate}
\item[(i)] $\phi$ is the identity automorphism;
\item[(ii)] $\phi$ is defined as in Equation \eqref{phi}, in the case that $n=4$ and  $(r,p,q)$ satisfy  the hypotheses of Proposition \ref{speaut4}.
\end{enumerate}
\end{thm}


\begin{remark} In fact, every automorphism of $G(r,p,q,n)$ has the form given in this theorem, provided $(r,p,q,n)$ is not $(1,1,1,6)$ or one of the twelve exceptions in Proposition \ref{characteristic-prop}. \end{remark}

In proving this theorem, it will be useful to note that if $j,k,z$ are as in Proposition \ref{alpha-thm} then the images of the generators $s_1,\dots,s_{n-1},s,t^p \in G(r,p,q,n)$ under the automorphism $\alpha_{j,k,z}$ are as follows:
  \begin{equation}\label{images}
   \alpha_{j,k,z}(s) = s^j,
\qquad
\alpha_{j,k,z}(t^p) = c^{kp/d_0}\cdot  t^{jp},
\qquad\text{and}
\qquad
\alpha_{j,k,z}(s_i) = z s_i.
\end{equation}

\begin{proof}
We  mimic the proof of \cite[Lemma 4.4]{Ma}.  
Since $\nu(N) = N$, Lemma \ref{Nchar} implies the automorphism $\bar \nu:S_n\rightarrow S_n$ is inner. Hence there exists $w\in S_n$ such that $\Ad(w^{-1})\circ \nu$ preserves the projection to $S_n$, so by   Lemma \ref{prepro} we have 
\[ \nu = \Ad(w) \circ \phi \circ \nu'\] where $\phi$ is either the identity or the automorphism in Proposition \ref{speaut4} and where $\nu' \in \Aut(G)$ satisfies
\[ \nu'(s_i)=z\( s_i,a_i e_i - a_i e_{i+1}\),\qquad\text{for all }i \in [n-1],\]
for some $z\in C(r,p,q,n)$ with $z^2=1$ and some integers $a_1,a_2,\dots,a_{n-1}$.

Now, as in the proof of \cite[Lemma 4.4]{Ma}, one can check that for the element 
\[
\begin{array}{c} 
y = \( \sum_{i=1}^n \sum_{j=i}^n a_j e_i, 1\) \in G(r,1,q,n)
\end{array}
\]
one has $y^{-1} \cdot \nu'(s_i) \cdot y = z s_i$ for each $i \in \{1,2,\dots,n-1\}$.  Therefore, if we let $\mu = \Ad(y^{-1}) \circ \nu' \in \Aut(G)$ then 
\[\nu = \Ad(w) \circ \phi \circ \Ad(y) \circ \mu  = \Ad\( w\cdot \phi(y) \) \circ \phi \circ \mu\] and $\mu(s_i) = zs_i$ for all $i$.
To complete our proof it suffices to show that $\mu = \alpha_{j,k,z}$ for some integers $j,k$.  Since $\mu$ already agrees with $\alpha_{j,k,z}$ on $s_1,\dots,s_{n-1}$, we need only  show that the images of the remaining generators $s$ and $t^p$ under $\mu$ have the same form as \eqref{images}.

To this end, first note that since $N$ is normal in $G$ and $\nu(N) = \phi(N) =  N$, we have $\mu(N) = N$. For some integers $x_i$ we may therefore write $\mu(s) = (1,x_1e _1 + \dots + x_n e_n)$. It is straightforward to work out that since $s_1ss_1 = s^{-1}$ 
and 
 $s_js s_j = s$ for each $j \in \{3,4,\dots,n-1\}$, 
we have 
 \[  x_1 +x_2 \equiv 2x_3 \equiv k \tfrac{r}{q} \modu r) \qquad\text{and}\qquad x_3\equiv x_4 \equiv \dots \equiv x_n \modu r)\]
 for some integer $k$.
It follows that 
%
%
\[ \mu(s) = z's^j,\quad\text{for some $z' \in C(r,p,q,n)$ with $(z')^2=1$ and some integer $j$}.\]
In particular, one can take $j = x_1-x_3$ and $z' = c^{x_3}$.
Applying $\mu$ to both sides of the identity $s_1s =  (s_2ss_2)^{-1}\cdot  s_1\cdot (s_2ss_2) $ shows  that in fact we must have $z' = 1$.

In a similar way, since $s_j t^p s_j =t^p$ for each $j\in \{2,3,\dots,n-1\}$, 
we must have
\[ \mu(t^p) = z''(t^{p})^{j'},\quad\text{for some $z'' \in C(r,p,q,n)$ and some integer $j'$}.\]
Applying $\mu$ to both sides of  $s^p = t^p \cdot s_1\cdot  t^{-p} \cdot s_1$
shows
that
$s^{jp} =   s^{j'p}$. Since $n\geq 3$ it follows that $j-j'$ is a multiple of $r/p$, so $t^{jp} = t^{j'p}$ and we may  assume $j=j'$. 
On the other hand, since $t^p$ has order $r/p$ in $G$, we must have  $z'' = c^k$ for some integer $k$ such that 
\begin{itemize}
\item $nk$ is a multiple of $p$ (this is equivalent to  $z'' \cdot t^{jp}\in G$);  
\item $rk/p$ is a multiple of $r/q$ (this is equivalent to $(z'')^{r/p} = 1$). 
\end{itemize}
These conditions imply that $k$ is a multiple of both ${p}/{\gcd(p,n)}$ and  
${p}/{\gcd(p,q)}$, which   means precisely that $k$ is a multiple of $p/d_0$ where $d_0 = \gcd(p,q,n)$.

We have thus shown that there are integers $j,k$ with $k$ a multiple of $p/d_0$ such that $\mu(s) = s^j$ and $\mu(t^p) = c^k \cdot t^{jp}$. 
Hence the images of $s$ and $t^p$ under $\mu$   
have  the form  (\ref{images}) which is what we needed to prove. 
\end{proof}

\section{Applications}

In this penultimate section we employ the preceding results to prove the remaining parts of Theorem \ref{thm3} in the introduction.
As usual we let $r,p,q,n$ be positive integers with $p$ and $q$ dividing $r$ and $pq$ dividing $rn$.
We also continue to  write $\tau$ for our standard automorphism $(\pi,x) \mapsto (\pi,-x)$ of $G(r,p,q,n)$. Observe that in terms of the automorphisms $\alpha_{j,k,z}$ in Proposition \ref{alpha-thm}, we have
$ \tau = \alpha_{-1,0,1}$ and $1 = \alpha_{1,0,1}$. 

\begin{lem}\label{vconginv}
  Let $n \geq 3$ and write $G=\qc$. Fix $\epsilon \in \{ -1,1\}$ and suppose $\nu \in \Aut(G)$ such that  
  the elements $\nu(h)$ and $h^{\epsilon}$ are conjugate in $G$ for all $h\in G$.
  \begin{enumerate}
  \item[(1)] For some $g \in G(r,1,q,n)$ and some $\alpha \in \{1,\tau\}$, we have $\nu = \Ad(g)\circ \alpha$.
  \item[(2)] If $\epsilon = 1$ and $(r,p,q,n)$ is  not one of the two exceptions \[\text{$(4,2,4,4)$ or $(4,4,4,4)$,}\] then $\alpha=1$.
  
  \item[(3)] If $\epsilon = -1$ and  $(r,p,q,n)$ is not one of the six exceptions \[\ba&\text{$(3,3,3,3)$ or $(6,3,3,3)$ or $(6,3,6,3)$ or $(6,6,3,3)$ or}\\&\text{$(4,2,4,4)$ or $(4,4,4,4)$,}\ea\] then $\alpha=\tau$.

  \end{enumerate}

\end{lem}

\begin{proof} 
We may assume $(r,p,q,n) $ is not $ (1,1,1,6)$, since the lemma holds in this case as every conjugacy class-preserving automorphism of $S_6$ is inner. Since $N$ is a normal subgroup, we must have $\nu(N) = N$, so
by Theorem \ref{vN=N} we know that $\nu = \Ad(g) \circ \phi \circ \alpha_{j,k,z}$ for some $g \in G(r,1,q,n)$ and some $j,k,z$ as in Proposition \ref{alpha-thm}, where $\phi$ is either trivial or the automorphism in Proposition \ref{speaut4}.
After  composing $\nu$ with an inner automorphism,
we  may further assume $g=t^i$ for an integer $i$.

Our next reduction is to note that even if $n=4$ and the conditions of Proposition \ref{speaut4} hold, then we still must have $\phi = 1$. For if $\phi$ is the automorphism in Proposition \ref{speaut4} 
then our hypothesis that $\nu(h)$ and $h^{\epsilon}$ are conjugate for all $h \in G$ fails for the element $h=s_3 \in S_4 \subset G$.  In detail, if $\phi \neq 1$ then, since we assume $g=t^i$, we have $\nu(s_3) = c^j \cdot  t^{r/2} \cdot s_3$ for some integer $j$, which is never conjugate to $(s_3)^{-1} = s_3$.
 
Letting $\alpha = \alpha_{j,k,z}$, we thus have $\nu = \Ad(t^i)\circ \alpha$.
Our next step is to show that $\alpha$ is either the identity automorphism or $\tau$. To this end we note that 
 $\nu$ coincides with $\alpha$ on $N$ since $\Ad(t^i)$ fixes $N$ pointwise,
so we have $\nu(s) = s^j$.
The elements $s$ and $s^{-1}$ are conjugate, and their conjugacy class in $G$ consists of all elements equal to $(e_{i_1} - e_{i_2},1)$ for two distinct integers $i_1,i_2\in \{1,2,\dots,n\}$. Therefore, for some  $i_1\neq i_2$ we must have 
\begin{equation}
\label{j=1}
 je_1 -je_2 -e_{i_1} + e_{i_2} \in \ZZ_r\spanning\left\{ \tfrac{r}{q} (e_1+e_2+\dots + e_n ) \right\}.
 \end{equation}
When $n\geq 5$ this is clearly only possible if the left expression is zero, which can occur only if $j \equiv \pm 1 \modu r)$.
On the other hand, one can check that if $n \in \{3,4\}$ then 
the given containment still holds only if $j \equiv \pm 1\modu r)$,
so $j \equiv \pm 1 \modu r)$
in all cases.

Similarly, we observe that $\nu (s_2) = \alpha(s_2) = z s_2$; recall that $z$ is an element of $C(r,p,q,n)$ with $z^2=1$.  The conjugacy class of $s_2^{-1} = s_2$ in $G$ consists of all elements of the form $\( (i_1, i_2), x e_{i_1} - x e_{i_2} \)$ where $x \in \ZZ_r$ and $i_1,i_2 \in \{1,2,\dots,n\}$ are distinct.
We can have $zs_2 = \((i_1, i_2),x e_{i_1} - x e_{i_2} \)$ for some such $x,i_1,i_2$   only if $x e_{i_1} - x e_{i_2}$ belongs to the right hand side of \eqref{j=1}, which is possible since $n\geq 3$ only if $x =0 $ and $z=1$.

Finally, we claim that we may assume $k =0$. If $p = r$ then we always have $\alpha_{j,k,z} = \alpha_{j,0,z}$ since $G$ is generated by $s$ and $s_1,\dots, s_{n-1}$ and the images of these generators under $\alpha$ have no dependence on $k$. On the other hand, if $p<r$ then 
\[\nu(t^p) = \(1, \tfrac{kp}{d_0}( 1,e_1+e_2 + \dots +e_n) + jp e_1\)\quad\text{where $d_0 = \gcd(p,q,n)$,}\] while the conjugacy class of $(t^p)^{\pm 1}$ in $G$ consists of all elements equal to $(1,\pm  pe_i)$ for some $i \in \{1,2,\dots,n\}$.
It is evident that an element of the form $(1,\pm pe_i)$ can only be equal to  $\nu(t^p)$ in $G$
if
$\frac{kp}{d_0}$ is a multiple of $\frac{r}{q}$, and in this case we again have $\alpha_{j,k,z} = \alpha_{j,0,z}$, so our claim follows.

We have thus  shown that either $\alpha = \alpha_{1,0,1} = 1$ or $\alpha = \alpha_{-1,0,1} = \tau$, which proves part (a).  
To prove part (2), we should assume $\epsilon = 1$ but $\alpha =\tau \neq 1$, 
and argue that $(r,p,q,n)$ is one of the two listed exceptions.
Likewise, to prove part (3), we should assume $\epsilon = -1$ but $\alpha = 1\neq \tau$, and argue that $(r,p,q,n)$ is  one of the six listed exceptions.
In the case that either of these pairs of assumptions holds, 
then necessarily $r \geq 3$, as otherwise $\tau = 1$, and we  have  $\nu(h) = h^{-\epsilon}$ for all $h \in N$, so every element of the diagonal subgroup $N = N(r,p,q,n)$ must be conjugate to its inverse in $G$.  We now show that these conditions hold only for a small number of quadruples $(r,p,q,n)$. Given this short of list of possibilities, it is a straightforward calculation (which we have carried out in {\sc{Magma}}) to check that only in the exceptional cases listed in parts (2) and (3) are $\Ad(t^i) \circ \tau(h)$ and $h$  always conjugate or $\Ad(t^i)(h)$ and $h^{-1}$ always conjugate. 

To this end, assume $r \geq 3$ and that each $h \in N$ is conjugate to $h^{-1}$ in $G$.  
Consider  the element $h = (1,e_1-2e_2 +e_3) \in N$. 
The conjugacy class of $h^{-1}$ in $G$ consists of all elements of the form $(1,-e_{i_1}+2e_{i_2} -e_{i_3})$ where $i_1,i_2,i_3$ are distinct elements of $\{1,2,\dots,n\}$. As $h$ is conjugate to $h^{-1}$ in $G$, for some choice of distinct indices $i_1,i_2,i_3$ we must have 
\begin{equation}\label{e} e_1 -2e_2 +e_3 +e_{i_1} -2 e_{i_2} + e_{i_3} \in \ZZ_r\spanning\left\{ \tfrac{r}{q} (e_1+e_2+\dots + e_n ) \right\}.\end{equation}
Since $r \geq 3$,   this containment is  impossible for $n>6$.
When $n \in \{3,4,5,6\}$, it is a routine but tedious exercise to determine, from  the finite list of  expressions which can occur as the right hand side of \eqref{e}, which values of $r,p,q$ allow \eqref{e} to hold. In particular, one finds that \eqref{e} holds only in the following cases:
\begin{itemize}
\item $n=3$ and $r \in \{3,6\}$ and $r/q\in \{1,2\}$. 
\item $n=4$ and $r=4$ and $q \in \{2,4\}$.
\item $n=5$ and $r=q=5$.
\item $n=6$ and $r=q=3$.
\end{itemize}
The exceptions given in parts (2) and (3) be one of these cases. As mentioned above, determining precisely which exceptions apply in each case is then a finite calculation, which is straightforward to carry out in a computer algebra system.
\end{proof}

Marin and Michel  \cite[Proposition 3.1]{MM} prove that the complex reflection groups $G(r,p,n)$ have no class-preserving outer automorphisms. (This is equivalent to the statement the the outer automorphism group of $G(r,p,n)$ acts faithfully on $\Irr(G(r,p,n))$, which is what the cited proposition in fact asserts.)
As is clear from Lemma \ref{5.1sub}, 
this property  significantly   simplifies the problem of classifying which finite complete reflection groups have GIMs. 

By contrast, the groups $G(r,p,q,n)$ can have class-preserving outer automorphisms:
the automorphism $\Ad(t^2)$ of $G(4,4,4,4)$ provides one example.
We can show, however, that   such automorphisms  only exist in a limited number of cases, and are almost always nearly inner.

\begin{prop} \label{no-class-prop}
The group $G(r,p,q,n)$ possesses a class-preserving outer automorphism
 only if  $(r,p,q,n) = (4,2,4,4)$ or if $n>2$ and 
$r\equiv p\equiv q\equiv n\equiv 2^i \modu 2^{i+1})$ for an integer $i>0$.
In these cases, provided $(r,p,q,n)$ is not $(4,2,4,4)$ or $(4,4,4,4)$, all class-preserving outer automorphisms are induced by inner automorphisms of $G(r,1,q,n)$.
 \end{prop}
 
 \begin{remark} Our familiar automorphism $\tau : (\pi,x)\mapsto(\pi,-x)$ provides a class-preserving outer automorphism of the groups $G(4,2,4,4)$ and $G(4,4,4,4)$ which is not induced by an inner automorphism of $G(4,1,4,4)$. 
 \end{remark}
 
\begin{proof} The proposition holds if $n=1$ since only the identity automorphism 
of an abelian group is class-preserving. The result follows from  Lemma \ref{no-class-lem} when $n=2$ and from
 Lemmas \ref{new4.3} and  \ref{vconginv} when $n\geq 3$.
\end{proof}

This proposition enables us to adapt several arguments in \cite{Ma} to prove the following result, which establishes part (2) of Theorem \ref{thm3}.

\begin{prop}\label{part2-prop} Assume $n>2$ and $\gcd(p,n) = 2$. If $q$ is odd, then $G(r,p,q,n)$ does not have a generalized involution model. 
\end{prop}

\begin{proof}
When $q=1$ the proposition coincides with
\cite[Lemmas 5.4 and 5.6]{Ma}. The arguments used to prove those results may be applied essentially without changes (once we incorporate a few facts proved in the present work) to the more general situation of this proposition.  We summarize the details as follows.

Assume $n>2$ and $\gcd(p,n) = 2$ and $q$ is odd, and write $G=G(r,p,q,n)$. Note that $r$, $p$, and $n$ are then all even.
 By Proposition \ref{no-class-prop}, $G$ has no class-preserving outer automorphisms, so by Lemma \ref{5.1sub}, to show that $G$ has no GIMs at all it suffices to show that $G$ has no GIMs with respect to our usual automorphism $\tau : (\pi,x) \mapsto (\pi,-x)$.

Assume $r/p$ is even.  
 Since $q$ is odd and $n$ is even,  $c^{r/2}$ is a nontrivial element of $G$. Furthermore,
  one can show as in \cite[Lemma 5.4]{Ma} that $c^{r/2}$ belongs to the commutator subgroup of the $\tau$-twisted centralizer 
of every generalized involution  $\omega \in \cI_{G,\tau}$. (It is not hard to see that if this holds for $q=1$ then it holds for any odd $q$.)
It follows  as in the proof of \cite[Lemma 5.4]{Ma} that $G$ can have no GIM with respect to $\tau$, since  the induced character of $G$ corresponding to any such model would contain $c^{r/2}$ in its kernel, contradicting the fact that the kernel of $\sum_{\psi \in \Irr(G)} \psi$ is $\{ 1\}$.

Alternatively assume $r/p$ is odd. The following facts then hold as in the proof of \cite[Lemma 5.6]{Ma}:
\begin{itemize}
\item The generalized involutions $1$ and  $\omega \eqdef (1,(1,-1,1,-1,\dots,1,-1))$  belong to disjoint $\tau$-twisted conjugacy classes in $G$ (in particular, $\omega \neq 1$).

\item The twisted centralizer $C_{G,\tau}(1)$ is generated by $s_1,s_2,\dots,s_{n-1} \in S_n$ together with  $s^{r/2} =(1,\tfrac{r}{2}(e_1-e_2)) \in N(r,p,q,n)$, and is isomorphic 
to the complex reflection group $G(2,2,n)$. 

\item Since we have  $z^{-1} \cdot \omega \cdot \tau(z) = c^{-1} \in C(r,p,q,n)$ for the element
$z \eqdef (1,(1,0,1,0\dots,1,0)) \in N(r,1,q,n),$  the automorphism $\Ad(z)$ of $G$ induces an isomorphism   $C_{G,\tau}(1) \xrightarrow{\sim} C_{G,\tau}(\omega)$.

\end{itemize}
Checking each of these claims is straightforward using our assumptions on $r,p,q,n$.
From these properties$-$in particular from the fact that any representative list of $\tau$-twisted centralizers in $G$ 
includes two conjugate subgroups isomorphic to $G(2,2,n)$$-$it follows by results of Baddeley \cite{B91}, exactly as in the proof of \cite[Lemma 5.6]{Ma}, that $G$ has no generalized involution models. 
\end{proof}

Finally, by applying part (c) of Lemma \ref{vconginv}, we may prove the following theorem establishing the remaining parts (3)-(5) of Theorem \ref{thm3} in the introduction.

\begin{thm}\label{main-thm}
Let $r,p,q,n$ be positive integers with $p$ and $q$ dividing $r$ and $pq$ dividing $rn$, and let $G =G(r,p,q,n)$. 
    
    \begin{enumerate}
    
    \item[(1)] If $\gcd(p,n) = 3$ then $G$ has a GIM if and only if $(r,p,q,n)$ is 
    \[ (3,3,3,3) \text{ or } (6,3,3,3) \text{ or } (6,6,3,3) \text{ or }(6,3,6,3).\]

\item[(2)] If $\gcd(p,n) = 4$ then $G$ has a GIM only if 
$r \equiv p \equiv q \equiv n \equiv 4 \modu 8).$

\item[(3)] If $\gcd(p,n) \geq 5$ then $G$ does not have a GIM.

    \end{enumerate}
\end{thm}

\begin{remark} Observe that case (2) asserts a necessary but  {not} sufficient condition. It remains an open problem to determine whether $G$ has a GIM when $\gcd(p,n) =4$ and $r \equiv p \equiv q \equiv n \equiv 4 \modu 8)$. Calculations show that $G(4,4,4,4)$  has a GIM but $G(12,4,4,4)$ and $G(12,4,12,4)\cong G(12,12,4,4)$ do not.
\end{remark}

\begin{proof}

Computer calculations show that  our assertions (1)-(3) hold if  $(r,p,q,n)$ is   one of the exceptions  listed in Lemma \ref{vconginv}.
We may therefore assume that $(r,p,q,n)$ is not one of these cases. Since (1)-(3) do not apply unless $\gcd(p,n) \geq 3$, 
we may also assume that 
$n\geq 3$.

Let $G = \qc$, and 
   suppose
   that $G$ has a GIM with respect to $\nu\in \Aut(G)$, so that 
   \[
   \sum_{\psi \in \Irr(G)} \psi(1) = |\cI_{G,\nu}|.
   \]
    By \cite[Proposition 2]{BG}, the elements   $\nu(h)$ and $h^{-1}$ are then conjugate in $G$ for all $h\in G$,
   so by Lemma \ref{vconginv} 
we 
must 
have
$\nu=\Ad(g)\circ \tau$ for some $g\in G(r,1,q,n)$. The idea of the rest of the proof is now simple: we just show that $|\cI_{G,\nu}| \leq |\cI_{G,\tau}|$, and then apply Theorem \ref{f}.

Heading in this direction, 
let $H = \{ \omega g : \omega \in G\}$ denote 
 the right coset of $G$ in $G(r,1,q,n)$ containing $ g$.
Since
 $\omega\in G$ satisfies $\omega\cdot \nu(\omega)=1$ if and only if $(\omega g)\cdot \tau(\omega g)=g\cdot \tau (g)$,
 we 
 then
 have
\[
   |\cI_{G,\nu}|=|\{h \in H :h\cdot \tau (h)=g\cdot \tau(g)\}|.
\]
The element $g\cdot \tau(g)$ belongs to the normal subgroup $N=N(r,p,q,n)$ as a consequence of the following argument.
Since  $\tau$ is an  involution, 
$\tau \circ \Ad(g) \circ \tau = \Ad(\tau(g))$. Since $\nu = \Ad(g)\circ \tau$ is also an involution, $ \Ad\( g\cdot \tau(g) \) =\nu^2= 1$, so 
 $g\cdot \tau(g)$ belongs to the center of $G$.  As $n \geq 3$, this implies $g \cdot \tau(g) \in N$ as claimed.

For each $\pi \in S_n$, let $\mathcal X_{\pi}$ denote the set of $h \in H$ with $|h| = \pi$ and $h \cdot \tau(h) = g\cdot \tau(g)$, so that $|\cI_{G,\nu}| = \sum_{\pi \in S_n} |\cX_\pi|$.
If we write $h \in \cX_\pi$ in the form $h = (\pi,x)$ then by definition
\[
 (\pi^2,\pi^{-1}(x)-x ) = g\cdot \tau(g).
 \]
Since the right hand side of this equation lies in $N$,
we must have $\pi^2=1$. 
Therefore if $(\pi,x),(\pi,y) \in \cX_\pi$ then $(\pi,x-y) \in \cI_{G,\tau}$, since automatically $(\pi,x-y) \in G$ as we have $\sum_i x_i \equiv \sum_i y_i \equiv \Delta(g) \modu p)$.
In light of this observation, there is evidently  an injective map $\bigcup_{\pi \in S_n} \cX_\pi \to \cI_{G,\tau}$, and so 
\[
   \sum_{\psi \in \Irr(G)} \psi(1) =
   |\cI_{G,\nu}| = \sum_{\pi \in S_n} |\cX_\pi| \leq |\cI_{G,\tau}|.\]
By Theorem \ref{f}, this inequality  holds only if it is an equality and either  $\gcd(p,n) \leq 2$ or  $\gcd(p,n) =4$ and $r\equiv p \equiv q \equiv n \equiv 4 \modu 8)$. Since we have assumed that $(r,p,q,n)$ is not one of the exceptions in (1), this completes our proof.
\end{proof}

\section{Conjectures} \label{last-sect}

The preceding results leave us with only a partial solution to the problem of determining which of the groups $G(r,p,q,n)$ haves GIMs. 
We close with some conjectures as to what a complete classification might look like. To explain our intuition behind these, we require briefly some additional terminology; as usual, continue to let $r,p,q,n$ be positive integers with $p$ and $q$ dividing $r$ and $pq$ dividing $rn$.

The irreducible representations of $G=G(r,p,q,n)$ are obtained from those of $G'=G(r,1,q,n)$ in the following way. Choose an irreducible representation $\rho$ of $G'$. The restriction $\Res_G^{G'}(\rho)$ is then the multiplicity-free sum of $k$ non-isomorphic irreducible representations $\rho_1,\rho_2,\dots,\rho_k$ of $G$. If $k>1$ then we say that each of the irreducible representations $\rho_i$ is \emph{split}. This notion is well-defined and depends only on the isomorphism class of $\rho_i$, because
  if $\rho$ and $\rho'$ are two irreducible representations of $G'$ then $\Res_G^{G'}(\rho)$ and $\Res_G^{G'}(\rho')$ are either isomorphic or have no isomorphic irreducible subrepresentations.
(This follows by Clifford theory since $G \vartriangleleft G'$ and $G'/G$ is cyclic; see \cite[\S6A]{Stembridge}.)
  
  We can say precisely when split representations (i.e., irreducible representations which are split) of $G$ exist.
  
  \begin{prop}
The group  $G(r,p,q,n)$ has no split representations if and only if  (i) $\gcd(p,n)=1$ or (ii) $\gcd(p,n)=2$ and $r\equiv p \equiv q \equiv n \equiv 2 \modu 4)$.
\end{prop}
\begin{proof} The following facts can be found in \cite[\S 6]{Ca1}.
The irreducible representations of $G(r,1,q,n)$  are indexed by $r$-tuples of integer partitions $(\lambda_0,\lambda_1,\dots,\lambda_{r-1})$ such that $\sum_i |\lambda_i|=n$ and 
such that $\sum_i i |\lambda_i|$ is divisible by $q$. The representation corresponding to such an $r$-tuple splits into more than one irreducible representation when restricted to $G(r,p,q,n)$ if and 
only if $\lambda_i=\lambda_{i+r/d}$ for all $i$ (where the indices are considered modulo $r$) for some  divisor $d> 1$ of $p$. 
Assume this condition holds for some $d$; 
it then follows that $d$ divides $\gcd(p,n)$ since 
\[
n=
\sum_{0\leq i < r} |\lambda_i| = d \sum_{0 \leq i < r/d} |\lambda_i|.
\]

It follows immediately that $G$ has no split representations if $\gcd(p,n) = 1$. Assume $\gcd(p,n) = 2$ and that the representation of $G(r,1,q,n)$ indexed by $(\lambda_0,\lambda_1,\dots,\lambda_{r-1})$ splits in $G$.
We must then have $\lambda_i = \lambda_{i+r/2}$ for all $i$, so 
\[ \sum_{0 \leq i < r} i |\lambda_i| = \sum_{0\leq i < r/2} (2i+\tfrac{r}{2}) |\lambda_i| \equiv \tfrac{r}{2} \sum_{0\leq i < r/2}|\lambda_i|   \equiv \tfrac{rn}{4}\modu 2).
\]
Since $q$ divides the left-most expression,  we cannot have $r\equiv p \equiv q \equiv n \equiv 2 \modu 4)$ as this would imply that the even number $q$ divides an odd number. Thus also in case (ii) $G$ has no split representations.

Assume now that  $\gcd(p,n)\geq 2$. 
 We wish to construct a split representation of $G$, so let $d>1$ be a prime divisor of 
$\gcd(p,n)$. The $r$-tuple of trivial partitions
\[
(\lambda_0,\lambda_1,\dots,\lambda_{r-1}) =
(\underbrace{\tfrac{n}{d},\emptyset,\ldots,\emptyset}_{r/d}, \underbrace{\tfrac{n}{d},\emptyset,\ldots,\emptyset}_{r/d},\ldots, \underbrace{\tfrac{n}{d},\emptyset,\ldots,\emptyset}_{r/d}).
\]
 indexes a split representation if   $d$ is odd, or if $d=2$ and either $n/2$ is even or $q$ is odd or $r/q$ is even,  since then $
\sum_{0\leq i < r} i |\lambda_i|=\frac{rn(d-1)}{2 d}$ is divisible by $q$.

Suppose we are in the remaining case that $\gcd(p,n)$ is a  nontrivial power of 2 and $d=2$, while $n/2$ is odd and $q$ is even and $r/q$ is odd. Then $\frac{rn}{4} \equiv \frac{q}{2} \modu q)$, and if $q/2$ is even   the $r$-tuple of trivial partitions 
\[ (\lambda_0,\lambda_1,\dots,\lambda_{r-1}) =
(\underbrace{\tfrac{n}{2}-\tfrac{q}{4},\tfrac{q}{4},\emptyset,\ldots,\emptyset}_{r/2},\underbrace{\tfrac{n}{2}-\tfrac{q}{4},\tfrac{q}{4},\emptyset,\ldots,\emptyset}_{r/2})
\]
indexes a split representation since $\sum_{0\leq i < r} i |\lambda_i|=\frac{nr}{4}+\frac{q}{2}$ is divisible by $q$.
This completes our proof because  if $q/2$ is odd then, since $n/2$ and $r/q$ are odd while $p$ is an even divisor of ${rn}/{q}$, 
we must have $r\equiv p \equiv q \equiv n \equiv 2 \modu 4)$ and in turn $\gcd(p,n) = 2$, and we have already considered this case.
\end{proof}

Following \cite{Ca2}, we say that a generalized involution of $G(r,p,q,n)$ with respect to the inverse transpose automorphism $\tau : (\pi,x)\mapsto(\pi,-x)$ is an \emph{absolute involution}. One checks that an element $\omega \in G(r,p,q,n)$ is an absolute involution if and only if (i) its preimages in $G(r,p,1,n)$ are  all symmetric matrices or (ii) its preimages in $G(r,p,1,n)$ are all antisymmetric matrices and $q$ is even. We say that $\omega$ is \emph{symmetric} or \emph{antisymmetric} according to these cases.
%
In analogy with the preceding proposition, we have this statement.

\begin{prop}
The group $G(r,p,q,n)$ has no antisymmetric absolute involutions if and only if (i) $q$ is odd or (ii) $n$ is 
odd or (iii) $r \equiv p \equiv q \equiv n \equiv 2 \modu 4)$. 
\end{prop}

\begin{proof}
Clearly $G=G(r,p,q,n)$ has no antisymmetric absolute involutions if $n$ or $q$ is odd, so assume both are  even. 
If $r,p,q,n$ are 
not all $\equiv 2 \modu 4)$ then either $p$ is odd or 4 divides $r$ or 4 divides $n$. In each of these cases we can find an integer $a$ such that $2a+\frac{rn}{4}$ is divisible by $p$, and  
the element $(\pi,x) \in G$ with
\[
\pi = (1,2)(3,4)\ldots(n-1,n) \in S_n \quad\text{and}\quad x = (a,a+\tfrac{r}{2},0,\tfrac{r}{2},\ldots, 0,\tfrac{r}{2}) \in (\ZZ_r)^n
\]
is then an antisymmetric absolute involution. On the other hand, if $r \equiv p \equiv q \equiv n \equiv 2 \modu 4)$ then $2a + \frac{rn}{4}$ is never a multiple of $p$, and it is straightforward to check that this implies that no absolute involution of $G(r,p,q,n)$ is antisymmetric.
\end{proof}
  
Combining the preceding two results gives us this corollary.

\begin{cor}\label{equiv-cor} 
The group $G(r,p,q,n)$ has  no split representations and no antisymmetric absolute involutions if and only if one of the following conditions holds: \begin{itemize} \item $\GCD(p,n) = 1$ and $q$ or $n$ is odd; \item $\GCD(p,n) = 2$ and $r\equiv p \equiv q \equiv n \equiv 2 \modu 4)$.\end{itemize}
\end{cor}

We have already shown that $G(r,p,q,n)$ has a GIM if $\gcd(p,n) = 1$ and $q$ or $n$ is odd.
The preceding corollary provides something of an explanation for this phenomenon, and so it seems reasonable to conjecture the following statement.

\begin{conj} $G(r,p,q,n)$ has a GIM if it has no split representations and no antisymmetric absolute involutions (i.e., if one of the conditions in Corollary \ref{equiv-cor} holds.)
\end{conj}

Computer calculations show that this is at least true in the interesting case $(r,p,q,n) = (2,2,2,6)$.
There is however another plausible conjecture which may explain the fact that $G(2,2,2,6)$ has a GIM.
Recall  that $G(r,p,q,n)$ is \emph{self-dual} if $G(r,p,q,n) \cong G(r,q,p,n)$. 
A necessary condition for the group $G(r,p,q,n)$ to have a generalized involution model (with respect to $\tau$) is that  the number of its absolute involutions  equal  the sum of the degrees of its irreducible characters. By Theorem \ref{f}, this occurs if and only if 
\[
 \gcd(p,n) \leq 2\qquad\text{or}\qquad \text{$\gcd(p,n) =4$ and $r \equiv p \equiv q \equiv n \equiv 4 \modu 8)$.}
\]
Based on Theorems \ref{thm1} and \ref{n=2-thm},
it might seem natural to conjecture that the group $G(r,p,q,n)$ has a GIM if it is self-dual and either of the two preceding conditions  hold. However, computations show that 
while $G(4,4,4,4)$ has a GIM, the groups $G(12,4,4,4)$ and $G(12,4,12,4) \cong G(12,12,4,4)$ do not. We are thus lead to the following modified conjecture; to this statement we do not yet have any counterexamples.

\begin{conj}  $G(r,p,q,n)$ has a GIM if it is  self-dual and $\gcd(p,n)\leq 2$.
\end{conj}

The preceding two conjectures seem to cover all the cases in which we know that $G(r,p,q,n)$ has a GIM, and so one is tempted to put forth the following much stronger conjecture. It seems intuitively desirable that a statement of this type hold, but admittedly we do not  have a lot of evidence to support it.

\begin{conj} If $(r,p,q,n)$ is not one of a finite number of exceptions, then $G(r,p,q,n)$ has a GIM  if and only if  (i) the group has no split representations and no antisymmetric absolute involutions
or (ii) the group is self-dual and $\gcd(p,n) \leq 2$.
\end{conj} 

This conjecture is appealing because it treats the cases $n=2$ and $n\neq 2$ simultaneously. (By Corollary \ref{equiv-cor} the conjecture coincides with Theorem \ref{n=2-thm} when  $n=2$.)
Our calculations show that among the groups $G(r,p,q,n)$ with order less than forty thousand, the conjecture holds provided $(r,p,q,n)$ is not one of the eight exceptions 
\[\ba&\text{$(3,3,3,3)$ or $(6,3,3,3)$ or $(6,3,6,3)$ or $(6,6,3,3)$ or} \\ &\text{$(4,1,2,2)$ or $(2,1,2,4)$ or $(4,4,4,4)$ or $(8,2,4,4)$.}\ea\] The groups corresponding to these cases all have GIMs but do not satisfy the conditions in the conjecture, and so seem to represent the ``right'' kind of exceptions.

Of course, beyond simply proving these conjectures, one desires an explanation for why self-duality or the lack of split representations and antisymmetric absolute involutions accounts for the existence of  generalized involution models

\end{document}